\documentclass{amsart}

\usepackage{amsmath,amsfonts,amssymb,amsthm,amscd,comment,euscript,calc}
\usepackage{mathbbol, stmaryrd, float}

\usepackage[all]{xy}
\usepackage{graphicx}
\usepackage{enumerate} 
\usepackage[colorlinks=true]{hyperref} 
\usepackage[usenames,dvipsnames]{xcolor}
\usepackage{tikz, tikz-cd, subfigure}

\usepackage{hyperref} 

\theoremstyle{plain}
\newtheorem{lem}{Lemma}[section]
\newtheorem{cor}[lem]{Corollary}
\newtheorem{prop}[lem]{Proposition}
\newtheorem{thm}[lem]{Theorem}

\theoremstyle{definition}
\newtheorem{ex}[lem]{Example}
\newtheorem{rem}[lem]{Remark}
\newtheorem{dfn}[lem]{Definition}

\newtheorem{conj}[lem]{Conjecture}
\newtheorem{claim}[lem]{Claim}

\newcommand{\RomanNumeralCaps}[1]
    {\MakeUppercase{\romannumeral #1}}

\newcommand{\so}{\mathfrak{so}}
\newcommand{\fsl}{\mathfrak{sl}}

\usetikzlibrary{decorations.markings}

\begin{document}

\title[Feigin-Stoyanovsky's principal subspaces]{Jet schemes, Quantum dilogarithm and Feigin-Stoyanovsky's principal subspaces}

\author{Hao Li}
\address{Yau Mathematical Sciences Center, Tsinghua, Beijing 100084}
\email{haoli2021@mail.tsinghua.edu.cn}
\author{Antun Milas}
\address{Department of Mathematics and Statistics, SUNY-Albany, NY 12222}
\email{amilas@albany.edu}

\maketitle

\begin{abstract}
We analyze the structure of Feigin-Stoyanovsky's principal subspaces of affine Lie algebra from the jet algebra viewpoint. For type $A$ level one principal subspaces,  we show that their shifted multi-graded Hilbert series can be expressed either using the quantum dilogarithm or as certain generating functions ``counting" finite-dimensional representations of $A$-type quivers. This notably results in novel fermionic character formulas for these principal subspaces. Moreover, our result implies that all level one principal subspaces of type $A$ are ``classically free" as vertex algebras. 

We also analyze infinite jet algebras associated to principal subspaces of affine vertex algebras $L_{1}(\mathfrak{so}_5)$, $L_{1}(\mathfrak{so}_8)$ and $L_1(\frak{g}_2)$. We derive a new character formula for the principal subspace of $L_1(\mathfrak{so}_5)$, proving that it is classically free, and present evidence that the principal subspaces of $L_1(\mathfrak{so}_8)$ and of $L_1(\frak{g}_2)$ are also classically free.
\end{abstract}

\section{Introduction}

\vskip 5mm
Associated varieties and associated schemes are important geometric object attached to vertex algebras.
They have been explored in prior works such as \cite{arakawa2018joseph,arakawa2019singular,arakawa2018arc,arakawa2016lectures,van2018chiral}, and have also garnered attention in theoretical physics, particularly due to their relevance to four-dimensional $\mathcal{N} = 2$ superconformal field theory (SCFT) \cite{beem2015infinite}.

A prominent result established by Arakawa and Li states that there is a surjective map $\psi$ from the infinite jet algebra $J_\infty(R_V)$ of a vertex algebra $V$ to the associated graded algebra ${\rm gr}(V)$ \cite{arakawa2012remark}, where $R_V$ denotes Zhu's Poisson algebra of $V$. Adopting the nomenclature of \cite{van2018chiral}, we say that $V$ is {\em classically free} when $\psi$ is injective. Another phrasing for injectivity is that $V$ is ``chiralization" of its Poisson algebra \cite{li2020some}. Vertex algebras
with this property are relatively common and we expect that they include many interesting rational vertex algebras and of course large families of irrational vertex algebras. It is relatively easy to find counterexamples to injectivity of $\psi$  \cite{arakawa2019singular}, \cite{van2018chiral}. 
Nevertheless, description of the kernel of the $\psi$-mapping is an intriguing and challenging problem. 
Very recently, Andrews, Van Ekeren and Heluani  \cite{van2020singular} unveiled a remarkable $q$-series identity that enabled them to describe the kernel of $\psi$ for the $c=\frac12$ Ising Virasoro vertex algebra.

As in \cite{li2020some} (see also \cite{jennings2020further}), our current focus pertains to associated schemes and infinite jet algebras associated with Feigin-Stoyanovsky's principal subspaces \cite{feigin1993quasi,milas2012lattice}.
Despite significant advancements in our understanding of principal subspaces and their characters with the help of vertex
algebra tools (cf. \cite{georgiev,capparelli2006rogers,calinescu2010vertex,primc}), we have encountered a limitation in the existing character formulas to 
establish the injectivity of the $\psi$ map. This primarily comes from the fact that most character formulas employ simple roots, rather than all positive roots (for some recent result on this subject see  for instance \cite{penn2018vertex}, \cite{butorac2020note}). In this paper, in conjunction with \cite{li2020some} and an upcoming sequel, we try to fill the gap between two distinct approaches to principal subspaces of affine vertex algebras by studying character formulas coming from all positive roots.

For instance, as demonstrated in the theorem below, the character of the principal subspace of $L_1(\mathfrak{sl}_n)$, obtained using positive roots (depicted on the left-hand side of equation (\ref{main})), is considerably more intricate when contrasted with the subspace obtained using simple roots (shown on the right-hand side of equation (\ref{main})).

\begin{thm}\label{thm1}
Let $A$ be the Cartan matrix $(\langle \alpha_{i},\alpha_{j} \rangle)_{1\leq i,j\leq n-1}$ of type $A_{n-1}$, $n \geq 2$, that is
$$ \langle \alpha_i,\alpha_j \rangle= \left\{ \begin{array}{cc}  2 \  &  \ i=j  \\ -1 \ &  \ |i -j|=1 \\ 0 \ & \ {\rm otherwise} \end{array} \right.$$
and
$${\bf m}=(m_{1,2},....,m_{n-1,n})=(m_{i,j})_{1 \leq i < j  \leq n} .$$

\begin{itemize}

\item[(a)]  We have $(q,x_1,...,x_n)$-series identities:


 \begin{align} \label{main} \sum_{{\bf m} \in \mathbb{N}_{\geq 0}^{n(n-1)/2}} \frac{{\displaystyle q^{B({\bf m})}} x_1^{\lambda_1} \cdots x_{n-1}^{\lambda_{n-1}}}{\displaystyle\prod_{1\leq i<j\leq n }(q)_{n_{i,j}}}=\displaystyle \sum_{\mathbf{k}=(k_{1},\ldots,k_{n-1})\in \mathbb{N}^{n-1}}\frac{q^{\frac12\mathbf{k}A\mathbf{k}^{\top}} x_1^{k_1} \cdots x_{n-1}^{k_{n-1}} }{(q)_{k_{1}}(q)_{k_{2}}\ldots (q)_{k_{n-1}}},\end{align}
where
\begin{align*}B({\bf m})=&\sum_{\substack{1\leq i_1 < j_1\leq n\\1\leq i_2 < j_2\leq n\\1\leq i_1 < i_2\leq n\\1\leq j_1 < j_2\leq n\\ j_{1}>i_{2}+1}  } m_{i_{1},j_{1}}m_{i_{2},j_{2}}+\sum_{\substack{1\leq i_1 < j_1\leq n\\1\leq i_2 < j_2\leq n\\1\leq i_{1}=i_{2}\leq n \\ 1\leq j_{1}\leq j_{2}\leq n }  } m_{i_{1},j_{1}}m_{i_{2},j_{2}}\\ &+\sum_{\substack{1\leq i_1 < j_1\leq n\\1\leq i_2 < j_2\leq n\\1\leq j_{1}=j_{2}\leq n \\  1\leq i_1<i_2\leq n}  } m_{i_{1},j_{1}}m_{i_{2},j_{2}}+\sum_{{j<i<i+1<j'}}m_{i,i+1}m_{j,j'}, \end{align*}
and $$\lambda_i=\displaystyle \sum _{ 1 \leq s\leq i \atop i < \ell \leq n}m_{s,\ell},$$



\item[(b)] we have
\begin{align}\label{main4} \sum_{{\bf m} \in \mathbb{N}_{\geq 0}^{n(n-1)/2}} \frac{{\displaystyle q^{B'({\bf m})}x_1^{\lambda_1} \cdots x_{n-1}^{\lambda_{n-1}}}}{\displaystyle\prod_{1\leq i<j\leq n }(q)_{m_{i,j}}}=\displaystyle \sum_{\mathbf{k}=(k_{1},\ldots,k_{n-1})\in \mathbb{N}^{n-1}}\frac{q^{\frac12\mathbf{k}A\mathbf{k}^{\top}} x_1^{k_1} \cdots x_{n-1}^{k_{n-1}}}{(q)_{k_{1}}(q)_{k_{2}}\ldots (q)_{k_{n-1}}},\end{align}
where
\begin{align*} B'({\bf m})=\sum_{\substack{1\leq i_1 < j_1\leq n\\1\leq i_2 < j_2\leq n\\1\leq i_1 \leq i_2\leq n\\1\leq j_1 \leq j_2\leq n\\j_{1}>i_2 }} m_{i_{1},j_{1}}m_{i_{2},j_{2}},\end{align*}
and $$\lambda_i=\displaystyle \sum _{ 1 \leq s\leq i \atop i < \ell \leq n}m_{s,\ell}.$$

\end{itemize}

\end{thm}




\begin{ex}
For  $\mathfrak{sl}_{3}$, we have $\lambda_1=m_{1,2}+m_{1,3},\lambda_2=m_{2,3}+m_{1,3}$.
Letting new variables $m=n_1$, $n=n_2$ and $n_1=m_{1,2},n_2=m_{1,3}$ and $n_3=m_{2,3}$, both equations (\ref{main}) and (\ref{main4}) give the following identity (also mentioned in \cite{feigin1993quasi}):
\begin{equation} \label{sl3}
\sum_{m,n \geq 0} \frac{q^{m^2+n^2-mn} x^m y^n}{(q)_m (q)_n}=\sum_{n_1,n_2,n_3 \geq 0} \frac{x^{n_1+n_3} y^{n_2+n_3} q^{n_1^2+n_2^2+n_3^2+ n_1 n_3+n_2 n_3}}{(q)_{n_1} (q)_{n_2} (q)_{n_3}}.
\end{equation}
\end{ex}

Notice that both the left-hand sides and right-hand sides of our identities in Theorem 1.1 are "stable" with respect to the rank $n$. In other words, if
we let $m_{\cdot,n}=0$ in the theorem, then all identities are valid for $\mathfrak{sl}_{n-1}$ (i.e. for type $A_{n-2}$).

In this work we are also concerned with more general $q$-series identities
\begin{equation} \label{generic}
\sum_{{\bf n} \in \mathbb{Z}^k_{\geq 0}} \frac{q^{ \frac12 {\bf n} A {\bf n}^\top } {\bf x}^{{\bf n}}}{(q)_{n_1} \cdots (q)_{n_k}}
=\sum_{{\bf m} \in \mathbb{Z}^{\ell}_{\geq 0}} \frac{q^{B ({\bf m})} {\bf x}^{U{\bf m}}}{(q)_{m_1} \cdots (q)_{m_\ell}},
\end{equation}
where ${\bf x}^{\bf n}:=x_1^{n_1} \cdots x_k^{n_k}$,  $A$ is a positive definite integral $k\times k$ matrix (e.g. Cartan matrix), $B({\bf m})$ is a quadratic form \footnote{Formula for $B({\bf m})$ 
is often complicated so we prefer not to use its matrix representation.} on $\ell$-dimensional space (in general $\ell \neq k$) and $U$ is a $k \times \ell $ matrix.
We say that the $q$-series identity (\ref{generic}) is {\em equivalent} to
\begin{equation} \label{generic2}
\sum_{{\bf n} \in \mathbb{Z}^k_{\geq 0}} \frac{q^{\frac12 {\bf n} \tilde{A} {\bf n}^\top } {\bf x}^{{\bf n}}}{(q)_{n_1} \cdots (q)_{n_k}}
=\sum_{{\bf m} \in \mathbb{Z}^\ell_{\geq 0}} \frac{q^{\tilde{B}( {\bf m})} {\bf x}^{ U{\bf m}}}{(q)_{m_1} \cdots (q)_{m_\ell}},
\end{equation}
if for every ${\bf n}$ and ${\bf m}$ as above with $U{\bf m}={\bf n}$, relation $$B ({\bf m}) - \frac12 {\bf n} A {\bf n}^\top=\tilde{B}({\bf m}) - \frac12 {\bf n} \tilde{A} {\bf n}^\top$$
holds. Under this condition, (\ref{generic}) implies (\ref{generic2}) and vice-versa as all ${\bf x}$ coefficients are equal. If we view (\ref{generic}) as an identity 
among two multi-graded Hilbert series of an algebra (or module), then {\em equivalence} of (\ref{generic}) and (\ref{generic2}) simply means that (\ref{generic2}) can be viewed 
as an identity between two Hilbert series of the same algebra (or module), but we performed a $q$-shift on each graded component of the original algebra with respect to the ${\bf x}$-grading.

The central idea within the proof of Theorem \ref{thm1} revolves around demonstrating that the equations (\ref{main}) and (\ref{main4}) are indeed equivalent with a pair of equations derived from 
from two distinct sources: on one hand, the quantum dilogarithm identities, and on the other hand, the enumeration of quiver representations.
Given that the right-hand side of equations (\ref{main}) and (\ref{main4}) is the known character formula of the {principal subspace} $W(\Lambda_0)$ of the level one affine vertex algebra $L_1(\fsl_n)$, as documented in references \cite{feigin1993quasi,calinescu2010vertex}, a direct consequence of Theorem \ref{thm1} are two novel character formulas for the principal subspace. Furthermore, building upon an insight from \cite{feigin2009pbw} concerning PBW filtrations, we carefully construct a sequence of filtrations applied to specific infinite jet algebras to facilitate an analysis of their Hilbert series. Employing this approach, we establish a result previously outlined in \cite[Theorem 5.7]{li2020some}, namely, the demonstration of the classical freeness of $W(\Lambda_0)$. Our outcome additionally provides a novel combinatorial monomial basis for $W(\Lambda_0)$, as detailed in reference \cite{li2020some}.


Our paper is structured as follows. In Section 2, we revisit standard definitions pertaining to the the $m$-jet algebra, arc algebra, affine vertex algebra and of the Feigin-Stoyanovsky principal subspace. For the sake of simplicity, we only consider principal subspaces for the vacuum representation. Section 3 introduces two distinct families of  filtration applied to specific infinite jet algebras, resulting in two inequalities among Hilbert series  (see Proposition \ref{filtra} and Proposition \ref{sec-ineq}). In the main part, Section 4 and Section 5, we show equivalences between the identities in (\ref{main}), (resp. (\ref{main4}))  and the identity derived from the quantum dilogarithm (resp. quiver representations). This equivalence is used to demonstrate classical freeness of $W(\Lambda_0)$  (see Corollary  \ref{main-cor}). Moving on to Sections 6 and 7, we employ the level one spinor representations for orthogonal series of affine Lie algebras \cite{feingold1985classical} to investigate the $C_{2}$-algebras of the principal subspace for $L_{1}(\so_5)$ and of $L_{1}(\so_8)$. Then we prove the following result:
\begin{thm} The principal subspace $W(\Lambda_0)$ of $L_1(\so_5)$ is classically free.
\end{thm}
As a corollary of this theorem, we prove a new $q$-series identity which can be viewed as a $B_2$ version of the pentagon identity for the quantum dilogarithm \footnote{The pentagon identity is usually associated to $A_2$ root system; see Section 4.1.} :
\begin{align*} \label{charge-id} &\sum_{\substack{r_1,r_2,r_3 \geq 0}} y_1^{r_1} y_2^{2r_3+r_2} \frac{q^{r_1^2+(r_2+r_3)^2+r_3^2-r_1(r_2+2 r_3)}}{(q)_{r_1}(q)_{r_2}(q)_{r_3}}= \sum_{\substack{n_1,n_2,n_3,n_4,n_{5} \geq 0}}\\  &
 y_1^{n_1+n_2+n_4}  y_2^{2n_1+2n_3+n_4+n_5}  \frac{q^{n_1^2+n_2^2+(n_3+n_5)^2+n_4^2+n_3^2+(2 n_3+n_5)n_1+n_4(n_1+n_2)+n_3 n_4}}{(q)_{n_1}(q)_{n_2}(q)_{n_3}(q)_{n_4}(q)_{n_5}}
 \nonumber  \end{align*}
 Subsequently, we propose a conjectured expression for the Hilbert series associated with the jet algebra of $R_{W(\Lambda_0)}$, where $W(\Lambda_0) \subset L_1(\so_8)$. Should this conjecture hold true, it would consequently establish the classical freedom of this particular principal subspace.
Concluding our study, in Section 8, we provide computational evidence that $W(\Lambda_0) \subset L_1(\frak{g}_2)$ (here $\frak{g}_2$ is the simple exceptional Lie algebra of type $G_2$) is also classically free.

{\bf Acknowledgments}:  The second named author would like to thank Robert Osburn for pointing out \cite{Andrews}  in connection to identities in Section 6. 
The second named author was partially supported by NSF grant DMS-2101844 and a Collaboration Grant for Mathematicians from the Simons Foundation 709563.


\section{preliminaries}
\subsection{Affine jet and arc algebras}
 We first define affine jet algebras.  As usual, let
$\mathbb{C}[x_1,x_2, \ldots, x_n]$ be the polynomial algebra in $x_i$, $1 \leq i \leq n$  and $f_{1},f_{2}, \ldots, f_{\ell}$ be some non-zero polynomials in this algebra. We will define the $m$-jet algebra of a finitely-generated commutative algebra: \begin{align*}R=\frac{ \mathbb{C}[x_1,x_2, \ldots, x_n]}{(  f_{1},f_{2}, \ldots, f_{\ell})}.\end{align*}

 Firstly, let us introduce new variables $x_{j,(-1-i)}$ for $i=0,\ldots,m$ and a derivation $T$ on \[\mathbb{C}[x_{j,(-1-i)}\;|\;0\leq i\leq m,\; 1\leq j\leq n], \] defined on generators \begin{equation*}
  T(x_{j,(-1-i)}) =
    \begin{cases}
      (-1-i)x_{j,(-i-2)} & \text{for $i\leq m-1$}\\
      0 & \text{for $i=m$}.\\
    \end{cases}
 \end{equation*}
Here we identify $x_{j}$ with $x_{j,(-1)}$. Set \begin{align*}R_{m}=\frac{\mathbb{C}[x_{j,(-1-i)}\;|\;0\leq i\leq m,\; 1\leq j\leq n]}{( T^{j}f_{i}|i=1,\ldots \ell ,\; m \geq j \geq 0) } ,  \end{align*}
 the algebra of $m$-jets of $R$.
The infinite jet algebra of $R$ is \begin{align*} J_{\infty}(R)&=\displaystyle\lim_{\underset{m}{\rightarrow}}R_{m}\\&=\frac{\mathbb{C}[x_{j,(-1-i)}\;|\;0\leq i,\; 1\leq j\leq n]}{( T^{j}f_{i}|i=1,\ldots \ell,\; j \geq 0) }.\end{align*}
The scheme $X_\infty=\displaystyle\lim_{\underset{m}{\leftarrow}} X_m$, where $X_m={\rm Spec}(R_m)$, is called the infinite jet scheme, or arc space, of $X={\rm Spec}(R)$.

We often omit "infinite" and call it jet algebra for brevity as we are not interested in the intermediate "finite" $m$-jet algebras. 
This construction is functorial and $J_\infty(\cdot)$ is right exact.
If $R$ is graded, then $J_\infty(R)$ is also graded, and we can define the Hilbert-(Poincar\'e) series of  $J_{\infty}(R)$ as: \[HS_{q}(J_{\infty}(R)):=\sum_{m \in  \mathbb{Z}} {\rm dim}(J_{\infty}(R)_{(m)}) q^m.\]

For later use, we denote the defining ideal of $J_\infty(R)$ by
\[ ( f_{1},\ldots,f_\ell)_{\partial}:=( T^{j}f_{i};i=1,\ldots,\ell, j\geq 0).\]

\subsection{Principal subspaces}

Let $V$ be a vector space.  A field is a formal series of the form $a(z)=\sum_{n\in \mathbb{Z}}a_{(n)}z^{-n-1}$, where $a_{(n)}\in $ End$(V)$ and for every $v\in V$ one has $$a_{(n)}v =0$$ for $n\gg 0$.
A vertex algebra is a vector space $V$, together with a distinguished vector called {\em vacuum}  $\mathbf{1} \in V,$ a linear derivation map $T \in {\rm End}(V)$, and the state-field correspondence map $$a\longmapsto Y(a,z)=\sum_{n\in \mathbb{Z}} a_{(n)}z^{-n-1},$$ satisfying the usual axioms of creation,  $T$-derivative properties, and Jacobi identity (for details see for instance \cite{lepowsky2012introduction}).
A vertex algebra $V$ is called commutative if $a_{(n)}=0$ for $n\geq 0$.

Let $\mathfrak{g}$ be a simple finite dimensional complex Lie algebra. Following \cite{lepowsky2012introduction}, we construct an affine untwisted Kac-Moody Lie algebra $\widehat{\mathfrak{g}}$ and its vacuum representation of level $k \neq -h^\vee$,
$$V_{k}(\mathfrak{g}) = U(\hat{\mathfrak{g}}) \otimes_{U(\mathfrak{g}[t]  ) \oplus \mathbb{C}c} \mathbb{C}_k, $$
which has a natural conformal vertex algebra structure. Here as usual $\mathfrak{g}[t] $ acts trivially on the $1$-dimensional space $\mathbb{C}_k$ and the central element $c$ acts as multiplication with $k$. We denote by $L_{k}(\mathfrak{g})$ the irreducible quotient of $V_{k}(\mathfrak{g})$, also a vertex algebra.

Next we define the {\em principal subspace}. We choose simple roots $\left\{\alpha_{1},\ldots,\alpha_{n}\right\}$ of $\frak{g}$ and denote the set of positive roots by  $\Delta^{+}$.  
Let $\mathfrak{n}_+:=\displaystyle \coprod_{\alpha\in \Delta^{+}} \mathbb{C} x_{\alpha}$, where $x_{\alpha}$ is the corresponding root vector of $\alpha$, and
$\widehat{\mathfrak{n}_{+}}=\mathfrak{n}_+ \otimes \mathbb{C}[t,t^{-1}]$ be its affinization.
For any affine vertex algebra $L_{k}({{\mathfrak{g}}})$, isomorphic to $L(k \Lambda_0)$ as a $\hat{\mathfrak{g}}$-module,
we define the (FS)-principal subspace of simple $\widehat{\mathfrak{g}}$-module  $L_{k}({{\mathfrak{g}}})$ as
$$W({k\Lambda_{0}}):= U(\widehat{\mathfrak{n}}_+) \cdot \mathbb{1},$$ where $\mathbb{1}$ is the vacuum vector. It is easy to see that this is a vertex algebra (without conformal vector). There is also a universal object
$$N({k\Lambda_{0}}):= U(\widehat{\mathfrak{n}}_+) \cdot \mathbb{1} \subset V_k(\mathfrak{g})$$
which is isomorphic (using PBW) to $U(t^{-1} \mathfrak{n}_+[t^{-1}] )$ as a graded vector space. From the construction we have a natural surjective map
$$N({k\Lambda_{0}}) \longrightarrow W({k\Lambda_{0}}),$$
with the kernel $I_{k \Lambda_0}$.  It is an open question to describe this ideal for every $k$. Clearly $I_{k \Lambda_0}=U(t^{-1} \mathbb{C}[t^{-1}] \mathfrak{n}_+){\bf 1} \cap I_{max}$, where
$I_{max}$ is the maximal ideal in $V_k(\mathfrak{g})$.

For $k=1$ and $\mathfrak{g}$ is of ADE type we have $W_L \cong W({\Lambda_{0}})$ where $W_{L} \subset V_L$, where $L$ is the root lattice, and $W_L$ is the principal space inside the lattice vertex algebra $V_L$ \cite{milas2012lattice}.
Principal subspaces were introduced in \cite{feigin1993quasi} and since then have been studied in many papers some by the authors; see \cite{butorac2014combinatorial,butorac2019principal,butorac2020note,calinescu2008intertwining,calinescu2010vertex,capparelli2006rogers,georgiev,kawasetsu,li2020some,LM2021,sadowski2015principal} and references therein.

Given a vertex algebra $V,$ we define its $C_{2}$-algebra as:\[R_{V}:=V/\langle a_{(-2)}b|a\in V, b\in V \rangle,\] where the multiplicative structure is given by \[\overline{a}\cdot \overline{b}=\overline{a_{(-1)}b}\] and the Lie bracket by (here $\overline{x}$ denotes the image of $x \in V$ in $R_V$):
 \[ \{ \overline{a} , \overline{b} \}=\overline{a_{(0)}b}.\]
 According to \cite{zhu1996modular}, these operations equip $R_{V}$ with  a Poisson algebra structure. Then $X_V:={\rm Spec}(R_V)$ is called the associated scheme of $V$.

\begin{dfn}


A commutative associative unital algebra $V$ is called a  vertex Poisson algebra if it is 
equipped with a linear operation,
$$V\rightarrow {\rm Hom}(V,z^{-1}V[z^{-1}]),\quad a\rightarrow Y_{-}(a,z)=\sum_{n\geq 0}a_{(n)}z^{-n-1},$$
such that
\begin{itemize}
    \item $(Ta)_{n}=-na_{(n-1)}$,
    \item $a_{(n)}b=\sum_{j\geq 0} (-1)^{n+j+1}\frac{1}{j!}T^{j}(b_{(n+j)}a),$
    \item $[a_{(m)},b_{(n)}]=\sum_{j\geq 0} \binom{m}{j}(a_{(j)}b)_{(m+n-j)},$
    \item  $a_{(n)}(b\cdot c)=(a_{(n)}b)\cdot c+b\cdot (a_{(n)}c),$
\end{itemize}
for $a,b,c\in V$ and $n,m \geq 0$.
\end{dfn}
Following  \cite{li2005abelianizing}, we can define a decreasing sequence of subspaces $\left\{F_{n}(V)\right\}$ of the
algebra $V$, where for $n\in \mathbb{Z}$, $F_{n}(V)$ is linearly spanned by the vectors $$u_{(-1-k_{1})}^{(1)}\ldots u_{(-1-k_{r})}^{(r)}\bf{1}$$
for $r\geq 1$, $u^{(1)},\ldots,u^{(r)}\in V,$ $k_{1},\ldots,k_{r}\geq 0$ with $k_{1}+\ldots+k_{r}\geq n.$ Then \[ V=F_{0}(V)\supset F_{1}(V)\supset\ldots\]such that  \begin{align*}
&u_{(n)}v\in F_{r+s-n-1}(V)\quad {\rm for} \quad u\in F_{r}(V),v\in F_{s}(V), r,s\in\mathbb{N},n\in \mathbb{Z},\\
&u_{(n)}v\in F_{r+s-n}(V)\quad {\rm for} \quad u\in F_{r}(V), v\in F_{s}(V),r,s,n\in \mathbb{N}. \end{align*}

\noindent The corresponding associated graded algebra 
${\rm gr}^{F}(V)=\coprod_{n\geq 0} F_{n}(V)/F_{n+1}(V)$ is a vertex Poisson algebra; sometimes we will omit the superscript $F$.

According to \cite{li2005abelianizing}, we know that $$F_{n}(V)=\left\{u_{(-1-i)}v|u\in V,i\geq 1,v\in F_{n-i}(V)\right\}.$$ In particular, $F_{0}(V)/F_{1}(V)=V/C_{2}(V)=R_{V}\subset gr^{F}(V)$. We can extend this embedding to a surjective map \[\psi:J_{\infty}(R_{V})\twoheadrightarrow gr^{F}(V).\]
Following \cite{arakawa2012remark}, $J_{\infty}(R_{V})$ has a unique vertex Poisson algebra structure such that

$$u_{(n)}v=\begin{cases} \left\{u,v\right\}, & \mbox{if } n=0 \\ 0, & \mbox{if } n>0 \end{cases} $$
for $u,v \in J_{\infty}(R)$.
The map $\psi$ is a vertex Poisson algebra epimorphism. 

We say that a vertex algebra is  {\em classically free} if  the map $\psi$ is an isomorphism.
\begin{ex} Not every vertex algebra is classically free. For instance, minimal Virasoro vertex algebras $V=L_{Vir}(c_{p,q},0)$, $2 \leq p<q$, $\gcd(p,q)=1$, are classically free if and only if 
$p=2$. This can 
be seen by looking at their characters \cite{FF,van2018chiral}.
\end{ex}

Returning to $L_k(\mathfrak{g})$ and its principal subspace $W(k \Lambda_0)$.
For these vertex algebras the filtration $F_n( \ \cdot \ ) $ is linearly spanned by
$$u_{(-1-k_{1})}^{(1)}\ldots u_{(-1-k_{r})}^{(r)}\bf{1},$$
$k_{1},\ldots,k_{r}\geq 0$ with $k_{1}+\ldots+k_{r}\geq n,$ where $u^{(i)} \in \frak{g}$ for $L_k(\mathfrak{g})$, and $u^{(i)} \in \frak{n}_+$ for $W(k \Lambda_0)$.
We clearly have
$R_{V_k(\mathfrak{g})}=\mathbb{C}[\mathfrak{g}]$, $R_{N(k \Lambda_0)}=\mathbb{C}[\mathfrak{n}_+]$, $R_{L_k(\mathfrak{g})}=\mathbb{C}[\mathfrak{g}]/ \overline{I_{max}}$, $R_{W(k \Lambda_0)}=\mathbb{C}[\mathfrak{n}_+]/\overline{I_{k \Lambda_0}}$
where $\overline{I_{max}}$ denotes the image of $I_{max}$ under the natural map $V \to R_V$, 
and ${I_{k \Lambda_0}} = U( \widehat{\mathfrak{n}_+}){\bf 1} \cap I_{max}$ \footnote{Unlike $I_{max}$, we in general do not know 
the generators of $I_{k \Lambda_0}$, $k \in \mathbb{N}$; this issue was discussed in several papers, e.g. \cite{calinescu2010vertex}.}. Taking projection gives
$$\overline{I_{k \Lambda_0}} \subset \mathbb{C}[\mathfrak{n}_+] \cap \overline{I_{max}}.$$
But in general, without any further assumption, we cannot prove the opposite inclusion.
For example, there is nothing that rules out relations of type $v=x_{\alpha_1}(-1)x_{\alpha_2}(-1)x_{\alpha_3}(-1)+x_{\alpha_1+\alpha_2+\alpha_3}(-2)h(-1) \in I_{max}$ where $h \in \frak{h}$, and $x_{\alpha_i} \in \frak{n}_+$ are some root vectors. This relation would imply $0 \neq \bar{v}=x_{\alpha_1}x_{\alpha_2} x_{\alpha_3} \in \overline{I}_{max} \cap  \mathbb{C}[\frak{n}_+] $ but $\overline{v}$ is not necessarily in $\overline{I}_{k \Lambda_0}$. Consequently, for general $k$,  we only have a map
$$R_{W(k \Lambda_0)} \rightarrow R_{L_k(\frak{g})},$$
but we additionally require $\overline{I_{k \Lambda_0}} = \mathbb{C}[\mathfrak{n}_+] \cap \overline{I_{max}}$ for the map to be an embedding.
We conclude this part with 
\begin{conj} If $L_k(\frak{g})$ is classically free, then $W(k \Lambda_0)$ is also classically free.
\end{conj}
The converse statement does not seem to hold.







\section{Principal subspace $W(\Lambda_0)$ of $\fsl_n$}

\subsection{Jet algebra $J_\infty(A)$} The structure of the principal subspace $W(\Lambda_0)$ of $\fsl_n$ was  already investigated by several authors \cite{feigin1993quasi,calinescu2008intertwining, calinescu2010vertex, georgiev, sadowski2015principal}.
According to \cite{feigin2011zhu} and \cite{li2020some} \footnote{The same result also follows from \cite{calinescu2010vertex} where the presentation of this principal subspace was given.}, the $C_{2}$-algebra $R_{W(\Lambda_0)}$ of the principal subspace of the affine vertex algebra $L_{1}(\fsl_{n})$ is isomorphic to \
\begin{equation} \label{initial}
A=\mathbb{C}[E_{i,j}|1 \leq i<j \leq n]/\underbrace{\langle \sum_{\sigma \in S_{2}} E_{i_1,j_{\sigma_{1}}}E_{i_2,j_{\sigma_{2}}}|j_{1}>i_2 \rangle}_{:=J}, 
\end{equation}
where we denote by $E_{i,j}$ the root vectors in $\mathfrak{n}_+$, and $1\leq i_{1}\leq i_{2}\leq n$, $1\leq j_{1}\leq j_{2}\leq n$ and $S_2$ is the symmetric group on two letters. Throughout this part we use the identification 
$$(E_{i,j})_{(-1)} =E_{i,j}$$
as elements of $J_\infty(R_{W(\Lambda_0)})$.
Inside the defining quadratic binomial ideal of $A$ we have the following three types of elements (it is helpful to use graphical
interpretation of these elements) :
\begin{itemize}
    \item[\textbf{Type \RomanNumeralCaps 1}]: $E_{i_{1},j_{1}}E_{i_{2},j_{2}}=-E_{i_{1},j_{2}}E_{i_{2},j_{1}},$ where $1\leq i_1 < j_1\leq n,\quad 1\leq i_2 < j_2\leq n,\quad 1\leq i_1 < i_2\leq n,\quad 1\leq j_1 < j_2\leq n,\quad j_{1}>i_2,\quad j_{1}>i_{2}+1.$

 \begin{tikzpicture}
\draw[gray, thick] (-12,-2) -- (-8,-2);
\draw[gray, thick] (-15,-1) -- (-10,-1);
\filldraw[black] (-15,-1) circle (2pt) node[anchor=south] {$i_{1}$};
\filldraw[black] (-10,-1) circle (2pt) node[anchor=south] {$j_{1}$};
\filldraw[black] (-12,-2) circle (2pt) node[anchor=south] {$i_{2}$};
\filldraw[black] (-8,-2) circle (2pt) node[anchor=south] {$j_{2}$};
\filldraw[black] (-11,-2) circle (2pt) node[anchor=south] {$i_{2}+1$};
\end{tikzpicture}
\vspace{2mm}

\item[\textbf{Type \RomanNumeralCaps 2}]:
$E_{i_{1},j_{1}}E_{i_{2},j_{2}}=0,$ where $1\leq i_1 < j_1\leq n,\quad 1\leq i_2 < j_2\leq n,\quad i_{1}=i_{2},\quad 1\leq j_{1}\leq j_{2}\leq n,$ or

$1\leq i_1 < j_1\leq n,\quad 1\leq i_2 < j_2\leq n,\quad 1\leq j_{1}=j_{2}\leq n,\quad 1\leq i_{1}< i_{2}\leq n.$

\begin{center}
 \begin{tikzpicture}
\draw[gray, thick] (-14,-2) -- (-10,-2);
\draw[gray, thick] (-14,-1) -- (-12,-1);
\filldraw[black] (-14,-1) circle (2pt) node[anchor=south] {$i_{1}$};
\filldraw[black] (-12,-1) circle (2pt) node[anchor=south] {$j_{1}$};
\filldraw[black] (-14,-2) circle (2pt) node[anchor=south] {$i_{2}$};
\filldraw[black] (-10,-2) circle (2pt) node[anchor=south] {$j_{2}$};
\end{tikzpicture}
\qquad
\begin{tikzpicture}
\draw[gray, thick] (-12,-2) -- (-10,-2);
\draw[gray, thick] (-14,-1) -- (-10,-1);
\filldraw[black] (-14,-1) circle (2pt) node[anchor=south] {$i_{1}$};
\filldraw[black] (-10,-1) circle (2pt) node[anchor=south] {$j_{1}$};
\filldraw[black] (-12,-2) circle (2pt) node[anchor=south] {$i_{2}$};
\filldraw[black] (-10,-2) circle (2pt) node[anchor=south] {$j_{2}$};
\end{tikzpicture}
\end{center}
\vspace{2mm}
\item[\textbf{Type \RomanNumeralCaps 3}]: $E_{i,i+1}E_{j,j'}=-E_{i,j'}E_{j,i+1},$ where $j<i<i+1<j'. $


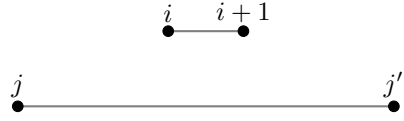
\vspace{2mm}
\begin{tikzpicture}
\draw[gray, thick] (-15,-1) -- (-16,-1);
\draw[gray, thick] (-13,-2) -- (-18,-2);
\filldraw[black] (-15,-1) circle (2pt) node[anchor=south] {$i+1$};
\filldraw[black] (-16,-1) circle (2pt) node[anchor=south] {$i$};
\filldraw[black] (-13,-2) circle (2pt) node[anchor=south] {$j'$};
\filldraw[black] (-18,-2) circle (2pt) node[anchor=south] {$j$};
\end{tikzpicture}

\end{itemize}
\begin{rem}

 Inside the jet algebra $J_\infty(A)$, Type \RomanNumeralCaps 1 relations generate further relations coming from $z$-coefficients of  $$E^{+}_{i_{1},j_{1}}(z)E^{+}_{i_{2},j_{2}}(z)+E^{+}_{i_{1},j_{2}}(z)E^{+}_{i_{2},j_{1}}(z)=0,$$  where $E^{+}_{i,j}(z):=\sum_{n\in\mathbb{Z}_{<0}}(E_{i,j})_{(n)}z^{-n-1}$ denotes the holomorphic part of the field $E_{i,j}(z)$. For
 Type \RomanNumeralCaps 2 and \RomanNumeralCaps 3 relations, we use similar notation.
In order to simplify presentation, we often write $a^+(z)b^+(z)=0$ to denote a whole family of relations obtained by taking all $z$-coefficients in the expansion.

\end{rem}

\subsection{Jet algebra $J_\infty(B)$}

We let $$P=\mathbb{C}[E_{i,j}|1 \leq i<j \leq n].$$ Let us consider another commutative algebra
\begin{align} \label{terminal} B=\mathbb{C}[E_{i,j}|1 \leq i<j \leq n]/I, \end{align} 
where ideal $I$ is generated by all Type \RomanNumeralCaps 2 relations and the left-hand sides of all relations of Type \RomanNumeralCaps 1 and Type \RomanNumeralCaps 3.  This way $I$ is a monomial ideal. This algebra and its infinite jet algebra are easier to analyze. We will show that $$HS_{q}(J_{\infty}(A))\leq HS_{q}(J_{\infty}(B)),$$
where the inequality among $q$-series has the obvious meaning: all $q$-coefficients on the left-hand side are less than or equal to  the corresponding coefficients on the right-hand side.
Eventually we will show that additionally $HS_{q}(J_{\infty}(B))={\rm ch}[W(\Lambda_0)]$ which together with $HS_{q}(J_{\infty}(A)) \geq {\rm ch}[W(\Lambda_0)]$ implies the classical freeness of $W(\Lambda_0)$ (see Section 4).


Let $i$ and $j$ be such that $n-2\geq i \geq  j \geq 1$.
Following an idea of E. Feigin on PBW filtrations in \cite{feigin2009pbw}, for $n \geq 3$, we introduce an increasing filtration $\{ G^{i,j}_s \}$, $ s \geq 0$,
 on the  jet algebra  $G^{i,j}=J_{\infty}({{B}_{i.j}})$, where ${B}_{i.j}$ is a  certain quotient algebra $\mathbb{C}[E_{a,b}|1 \leq a<b \leq n]/\mathbb{I}_{i.j}$ by an ideal generated by quadratic binomials (of degree 2). Roughly speaking, our procedure starts with the initial algebra $A=P/J$ and consists of several steps such that in the terminal algebra $B=P/I$, ideal $I$ is monomial. This resulting algebra is easier to analyze. In particular, we can explicitly compute the Hilbert series of $J_\infty(B)$ using methods of vertex algebra \cite{li2020some}.


{\em Construction of $gr^{i.j}$.} Next we define filtrations and corresponding associated graded algebras. Fix a pair $(i,j)$ as above. Suppose that we already constructed algebra $\mathcal{A}=P/\mathbb{I}$, a quotient algebra of $P$, and hence $J_\infty(\mathcal{A})$. We define the filtration $\{ G^{i,j}_s \}_{s \geq 0}$ on $J_{\infty}({\mathcal{A}})$ by first defining a filtration on $J_{\infty}(P)$. Let $G^{i,j}_{0}(P)$  be diffferentially  generated by $$E_{i,i+2},E_{i-1,i+2},\cdots, E_{j,i+2},$$ 
that is 
\begin{align*}& G^{i,j}_{0}(P) :=\langle E_{m,n} : (m,n) \in \mathcal{E}_{i,j} \rangle_\partial \\ 
& = {\rm span}\left\{(E_{i,i+2})^{\ell_{1}}_{(k_{1})}(E_{i-1,i+2})^{\ell_{2}}_{(k_{2})}\cdots (E_{j,i+2})^{\ell_{i-j+1}}_{(k_{i-j+1})}|  \ell_{r} \in \mathbb{N}_{\geq 0}, k_{m} \in \mathbb{Z}_{<0}\right\},
\end{align*} 
where for convenience we let $$\mathcal{E}_{i,j}=\left\{(i,i+2), (i-1,i+2),\cdots, (j,i+2)\right\}.$$ 
Then we recursively  let ($s \geq 1$): \[G^{i,j}_{s}(P):={\rm span}\left\{(E_{u,u'})_{(i)}v|i\leq -1, ({u,u'})\notin \mathcal{E}_{i,j}, v\in G^{i,j}_{s-1}\right\}+G^{i,j}_{s-1}(P).\] 
This clearly defines a filtration on $J_{\infty}(P)$. To define filtration on $J_{\infty}(\mathcal{A})=J_\infty(P)/\mathbb{I}_\partial$ we use the standard construction where we let 
$G^{i,j}_s:=G^{i,j}_s(P)/( G^{i,j}_s(P) \cap \mathbb{I}_\partial)$ and thus $G^{i,j}_s/G^{i,j}_{s-1} \cong G^{i,j}_{s}(P)/{(G^{i,j}_{s-1}(P)+G_s^{i,j}(P) \cap \mathbb{I}_\partial})$.

Then the associated graded algebra of $J_{\infty}(\mathcal{A})$, for a fixed pair $(i,j)$,  
is defined as \[gr^{i.j}(J_{\infty}(\mathcal{A})):= { \bigoplus_{s \geq 0}} \frac{G^{i,j}_{s}(P)}{(G^{i,j}_{s-1}(P)+G^{i,j}_s(P) \cap \mathbb{I}_\partial)}.\]  
Observe that we have a natural homomorphism from $P$ to $gr^{i.j}(J_{\infty}(\mathcal{A}))$ sending $E_{i,j}$ to its coset. The corresponding kernel will be denoted by $\mathbb{I}_{i.j+1}$, if $i >j$, and $\mathbb{I}_{i+1,1}$ if $i=j$.
\color{black}


Since we are dealing a whole family of algebras and filtrations, next we explain every step in the construction including the initial and the terminal step.
\begin{itemize}
 \item[\textbf{Step $1.1$}:] At the initial step we let $B_{0.0}:=A=P/J$. Using the prescription above, we get an algebra homomorphism from $P$ to $gr^{1.1}(J_{\infty}(A))$, where we  send $E_{i,j}$ to its coset. We restrict the map to $P$
 and denote the kernel (of the restriction) by $\mathbb{I}_{1.1}$, and define a new algebra \[B_{1.1}:=\mathbb{C}[E_{i,j}|1 \leq i<j \leq n]/\mathbb{I}_{1.1}.\]
\item[\textbf{Step $2.1$}:] Define \[B_{2.1}:=\mathbb{C}[E_{i,j}|1 \leq i<j \leq n]/\mathbb{I}_{2.1},\] where $\mathbb{I}_{2.1}$ is the kernel of the algebra homomorphism from $P$ to $gr^{2,1}(J_{\infty}(B_{1.1}))$, where $B_{1.1}$ was constructed in Step 1.1.
    \item[\textbf{Step $2.2$}:] Define \[B_{2.2}:=\mathbb{C}[E_{i,j}|1 \leq i<j \leq n]/\mathbb{I}_{2.2},\] where $\mathbb{I}_{2.2}$ is the kernel of the map defined in the same manner as above from $P$ to $gr^{2.2}(J_{\infty}(B_{2.1}))$.
\end{itemize}

We proceed in this fashion so that the input $B$-algebra needed at Step $i.j$ (with $1<j \leq i $) is the one constructed in Step $i.j-1$ and the $B$-input algebra needed for Step $i.1$ is the one constructed at Step $i-1.i-1$. More explicitly:
\begin{itemize}

\item[\textbf{Step $i.1$}:] Define \[B_{i.1}:=\mathbb{C}[E_{i,j}|1 \leq i<j \leq n]/\mathbb{I}_{i.1}\] where $\mathbb{I}_{i.1}$ is the kernel of the map from $P$ to $gr^{i.1}(B_{i-1.i-1})$.

\item[\textbf{Step $i.j$}] ($1<j \leq i$): Define \[B_{i.j}:=\mathbb{C}[E_{i,j}|1 \leq i<j \leq n]/\mathbb{I}_{i.j}\]  where $\mathbb{I}_{i.j}$ is the kernel of the map from $P$ to $gr^{i.j}(B_{i.j-1})$.

\end{itemize}
This procedure terminates  at Step $n-2.n-2$, where we reached a desired algebra $B_{n-2.n-2}.$

The maps that we constructed extend to the corresponding arc algebras. In other words, we have:
\begin{claim}For every $i$ and $j$ as above, we have the surjective homomorphism: \begin{align}\begin{cases} \displaystyle J_{\infty}(B_{i.j}) \twoheadrightarrow gr^{i.j}(J_{\infty}(B_{i.j-1}))& \mbox{if } i \geq j>1, \\  J_{\infty}(B_{i.j}) \twoheadrightarrow gr^{i.j}(J_{\infty}(B_{i-1.i-1})) & \mbox{if } j=1. \\\end{cases} \end{align}   \end{claim}
\begin{proof}
We only prove the Claim at Step $1.1$. The proof at other steps is analogous using induction. 
It is clear from the construction that $gr^{1.1}(J_\infty(A))$ is generated by $(E_{i,j})_{(-k)}$, so there is a natural algebra epimorphism $\phi$ from $J_\infty(P)$ to $gr^{1.1}(J_{\infty}(A))$ (here $B_{0,0}=A$) and the restriction of this map to $P$ has the kernel $\mathbb{I}_{1.1}$, $B_{1.1}=P/\mathbb{I}_{1.1}$. 
So we only have to argue that $\phi((\mathbb{I}_{1.1})_{\partial})=0$, which implies that  the surjective map $\phi$ is actually defined on $J_\infty(B_{1.1})$. Let $v \in \mathbb{I}_{1.1}$. 
From the construction, we may assume that there exists $s\in \mathbb{N}$ such that $v   \in G^{1.1}_s$ and $v \in G^{1.1}_{s-1} + \mathbb{I}_\partial \cap G^{1.1}_{s}$. But, by the construction, $G^{1.1}_s$ and $\mathbb{I}_\delta$ are invariant under the derivation map $T$, so we have $\phi(T (v))=0$.
\end{proof}

{\bf Hilbert series.} Although $gr^{i.j}(J_\infty(P))$ has a new grading (with infinite dimensional graded subspaces) we will consider it 
as a graded algebra with respect to the standard grading inherited from $J_\infty(P)$, that is $deg((E_{i,j})_{(-k)})=k$. Since all $B$-ideals starting with $J$ are homogeneous 
then $gr^{i.j}(J_\infty(B_{i.j-1}))$ and $gr^{i.1}(J_\infty(B_{i-1.i-1}))$ are also homogeneous so we have a well-defined Hilbert series and relations
\begin{align}\begin{cases} \displaystyle HS_{q}(gr^{i.j}(J_{\infty}(B_{i.j-1})))=HS_{q}(J_{\infty}(B_{i.j-1}))& \mbox{if } j-1\geq 1,\\  HS_{q}(gr^{i.j}(J_{\infty}(B_{i-1.i-1})))=HS_{q}(J_{\infty}(B_{i-1.i-1})) & \mbox{if } j=1. \\\end{cases}.\end{align} 
Claim 3.2 gives a sequence of surjective homomorphisms from which we obtain a chain of inequalities among their Hilbert series:
\begin{align}\label{chain} HS_{q}(J_{\infty}(B_{n-2.n-2}))\geq \cdots \geq HS_{q}(J_{\infty}(B_{n-2.1}))\geq \cdots \geq HS_{q}(J_{\infty}(A)).\end{align}

We illustrate how everything in this part works on a simplest non-trivial example.

\begin{ex} For $\fsl_3$, we have $A=B$ so there is nothing to prove, so we consider the Lie algebra $\fsl_{4},$ that is $n=4$. Then we have \begin{align*} A=\mathbb{C}[E_{1,2},E_{1,3},E_{1,4},E_{2,3},E_{2,4},E_{3,4}]/J, \end{align*} where \begin{align*} J=\langle & E_{1,2}^{2},E_{1,3}^{2},E_{1,4}^{2},E_{2,3}^{2},E_{2,4}^{2},E_{3,4}^{2}, E_{1,2}E_{1,3}, E_{1,2}E_{1,4},E_{1,3}E_{1,4}, E_{2,3}E_{2,4},\\ & \color{black} E_{1,3}E_{2,4}+E_{1,4}E_{2,3} \color{black}, E_{1,3}E_{2,3},E_{1,4}E_{2,4},E_{1,4}E_{3,4},E_{2,4}E_{3,4}\rangle.
\end{align*}
In this case, the relevant algebras are $B_{1.1}$, $B_{2.1}$ and $B_{2.2}$ constructed in Step 1.1, Step 2.1, and Step 2.2, respectively.

\noindent \textbf{Step 1.1}   
    Here we write down the spaces $G^{1,1}_{0} \subset J_\infty(P)$ and $G^{1,1}_{1} \subset J_\infty(P)$ explicitly for convenience: \begin{align*}G^{1,1}_{0}=\langle E_{1,3} \rangle_\partial={\rm span}\left\{(E_{1,3})^{\ell}_{(k)}| \ell \in \mathbb{N} ,k\in \mathbb{Z}_{<0}\right\},\end{align*}
    \begin{align*}G^{1,1}_{1}={\rm span}\left\{ \right.&(E_{1,2})_{(k_{1})}(E_{1,3})^{\ell}_{(k_{2})},\,(E_{1,4})_{(k_{1})}(E_{1,3})^{\ell}_{(k_{2})},\,(E_{2,3})_{(k_{1})}(E_{1,3})^{\ell}_{(k_{2})},\,\\ &(E_{2,4})_{(k_{1})}(E_{1,3})^{\ell}_{(k_{2})},\,(E_{3,4})_{(k_{1})}(E_{1,3})^{\ell}_{(k_{2})},\ | \ell \in \mathbb{N} ,k_{1},k_{2}\in \mathbb{Z}_{<0}\left. \right\}+G^{1,1}_{0}.\end{align*}
We claim:
\begin{align*} 
& B_{1.1}=P/\mathbb{I}_{1.1} \ \ {\rm and} \ \   I \subset \mathbb{I}_{1.1} \\  
& I= \langle E_{1,2}^{2},E_{1,3}^{2},E_{1,4}^{2},E_{2,3}^{2},E_{2,4}^{2},E_{3,4}^{2}, E_{1,2}E_{1,3}, E_{1,2}E_{1,4},E_{1,3}E_{1,4}, E_{2,3}E_{2,4}, \color{black} E_{1,4}E_{2,3} \color{black}, \\
&  E_{1,3}E_{2,3},E_{1,4}E_{2,4},E_{1,4}E_{3,4},E_{2,4}E_{3,4}\rangle.
\end{align*} 
To see that, we first notice that  in $gr^{1.1}(J_{\infty}(A))$ all monomial relations for $J$ are present but also $E_{1,3}E_{2,4}=-E_{1,4}E_{2,3}$ yields a new relation. More precisely,  at Step $1.1$ we have $E_{1,3}E_{2,4}\subset G^{1,1}_{1}$ and $E_{1,4}E_{2,3}\subset G^{1,1}_{2}\setminus G^{1,1}_{1} .$ Thus, the relation $E_{1,3}E_{2,4}=-E_{1,4}E_{2,3}$ in $J$ gives us $E_{1,4}E_{2,3}=0$ in $gr^{1.1}(J_\infty(A))$. 
 Therefore we have a surjective map from $J_\infty(B_{1.1})$ to $gr^{1.1}J_\infty(A)$ and thus $HS_{q}(J_{\infty}(A))\leq HS_{q}( J_{\infty}(B)).$
  Observe that already at this stage we established a surjective map 
 $$J_\infty(B) \twoheadrightarrow J_\infty(A)$$
and thus $HS_{q}(J_{\infty}(A))\leq HS_{q}( J_{\infty}(B)).$
 \text{Step 2.1 and Step 2.2} do not bring any practical improvements. We again obtain $$J_\infty(B) \twoheadrightarrow J_\infty(B_{2.2}),$$
 with $I \mathbb \subset \mathbb{I}_{2.2}$. We  can actually prove \[ B_{2.2} \cong B_{2.1} \cong B_{1.1} \cong B, \]  but this stronger result is not needed here. 
 \end{ex}

What we observed in the previous example for $n=4$ holds for every $\fsl_n$.

\begin{prop}\label{filtra}
For every $n \geq 3$, we have \[HS_{q}(J_{\infty}(A))\leq HS_{q}( J_{\infty}(B)).\]
\end{prop}

\begin{proof} 



According to (\ref{chain}), we know that $$ HS_{q}( J_{\infty}(B_{n-2.n-2}))\geq HS_{q}(J_{\infty}(A)).$$ In order to prove the statement, it is sufficient to show that the defining ideal of $B$, i.e., $I$, belongs to $\mathbb{I}_{n-2.n-2}.$ Indeed, if it is true, we would have $$ HS_{q}( J_{\infty}(B))\geq HS_{q}( J_{\infty}(B_{n-2.n-2}))$$ since both $B$ and $B_{n-2.n-2}$ are quotient polynomial algebras with the same generators of the same degree.  

{\em Claim:} $I\subset \mathbb{I}_{n-2.n-2} $.

We first consider Type \RomanNumeralCaps 2 relations, coming from quadratic monomial relations. Suppose that $v=E_{i,j} E_{r,s} \in J$. These relations are present at every step of the construction because there is $k \leq 2$ such that $v \in G^{k}_{1.1}(P)$, $v \notin G^{k-1}_{1.1}(P)$ and $v \in G^{k}_{1.1}(P) \cap J$. The rest of the argument follows inductively and therefore $v \in \mathbb{I}_{n-2.n-2}$.

Next we consider Type \RomanNumeralCaps 1 relations in $A$, i.e.,
 $E_{i_{1},j_{1}}E_{i_{2},j_{2}}=-E_{i_{1},j_{2}}E_{i_{2},j_{1}},$ where $1\leq i_1 < j_1\leq n,\quad 1\leq i_2 < j_2\leq n,\quad 1\leq i_1 \leq i_2\leq n,\quad 1\leq j_1 \leq j_2\leq n,\quad  j_{1}>i_{2}+1.$ Following above procedure, we see that these Type \RomanNumeralCaps 1 relations belong to $\mathbb{I}_{i,j}$ until the step $(j_{1}-2).(i_{1}+1)$. Indeed,  at that  step,  $G^{j_{1}-2,i_{1}+1}_{0}$ is generated by $$E_{j_{1}-2,j_{1}}, E_{j_{1}-3,j_{1}},\cdots,E_{i_{1}+1,j_{1}}.$$ Note \[E_{i_{1},j_{1}}E_{i_{2},j_{2}}\subset G^{j_{1}-2,i_{1}+1}_{2}\setminus G^{j_{1}-2,i_{1}+1}_{1},\, \text{while}\,\,\, E_{i_{1},j_{2}}E_{i_{2},j_{1}}\subset G^{j_{1}-2,i_{1}+1}_{1}.\] Therefore, we get  $E_{i_{1},j_{1}}E_{i_{2},j_{2}}=0$ in $B_{j_{1}-2.i_{1}+1}$ at the step $(j_{1}-2).(i_{1}+1)$, and these relations also remain true in $B_{n-2.n-2}.$


For Type \RomanNumeralCaps 3 relations of $A$, i.e., $E_{i,i+1}E_{j,j'}=-E_{i,j'}E_{j,i+1}$ where $j<i<i+1<j',$ they give us  $$E_{i,i+1}E_{j,j'}=0$$ in $B_{i-1.1}$ at the step $(i-1).1$ since  we have $$E_{i,i+1}E_{j,j'}\subset G^{i-1,1}_{2}\setminus G^{i-1,1}_{1},\,\, \text{and}\,\,\,\, E_{i,j'}E_{j,i+1}\subset G^{i-1,1}_{1}.$$  Thus, these Type \RomanNumeralCaps 3 relations   in  $B_{n-2.n-2}$ become $E_{i,i+1}E_{j,j'}=0.$

From above arguments and the definition of algebra $B$, we see that all quotient relations of $B$ are true in $B_{n-2.n-2}$. Hence, we obtain the result.
\end{proof}



\subsection{Jet algebra $J_\infty(H)$}
If instead of taking the left-hand sides of Type \RomanNumeralCaps 3 relations for the generators of the quotient ideal of $B,$ we take the {\em right-hand sides} of Type \RomanNumeralCaps 3 relations, we get another commutative algebra \begin{align*}H:=\mathbb{C}[E_{i,j}|1 \leq i<j \leq n]/\mathcal{I}. \end{align*} 

Following a similar filtration procedure as in Section 3.2, we introduce another family of filtrations, $\mathcal{G}^{i,j}$ which are defined in the same manner as $G^{i,j}$ just with a different initial generating set for $\mathcal{G}^{i,j}_{0}$ (see below). We abuse the notation slightly here and still use $gr^{i.j}(J_{\infty}(\mathcal{H}))$ to denote the associated graded algebra of $J_{\infty}(\mathcal{H})$ but now with respect to the filtration $\mathcal{G}^{i,j}$.
The range for $i$ and $j$ in this case is: $ 1 \leq i \leq n-2$, $1 \leq j \leq i+1$ and 
$$\mathcal{E}'_{i,j}=\left\{(i+1,i+2), (i,i+2),\cdots, (j,i+2)\right\},$$
so that $$\mathcal{G}^{i,j}_0(P)=\langle E_{i,j} : (i,j) \in \mathcal{E}'_{i,j} \rangle_\partial.$$

\begin{ex}
For example for $n=4$, we have $ 1 \leq i \leq 2$, $1 \leq j \leq i+1$, and five steps in the construction: Steps 1.1 and 1.2, and Steps 2.1, 2.2, and 2.3. 
The corresponding generators for $\mathcal{G}_{0}^{i,j}(P)$ are given by 
\begin{align*}
 \textbf{Step $1.1$} \quad \mathcal{G}^{1,1}_{0}:  & \  E_{2,3}, \quad E_{1,3} \\
     \textbf{Step $1.2$} \quad  \mathcal{G}^{1,2}_{0}:  & E_{2,3}, \\
     \textbf{Step $2.1$} \quad \mathcal{G}^{2,1}_{0}:  & \ E_{3,4}, \quad E_{2,4}, \quad E_{1,4} \\
     \textbf{Step $2.2$} \quad  \mathcal{G}^{2,2}_{0}:  & \  E_{3,4}, \quad E_{2,4} \\
     \textbf{Step $2.3$} \quad  \mathcal{G}^{2,3}_{0}:  & \  E_{3,4} .\\
    \end{align*}
 \end{ex}   
    
In general, the construction starts as before with $A=P/J$ and the jet algebra $J_\infty(A)$. Then in Step 1.1 we construct a commutative algebra
 $$H_{1.1}:=\mathbb{C}[E_{i,j}|1 \leq i<j \leq n]/\mathbb{T}_{1.1}, $$ where  $\mathbb{T}_{1.1}$ is the kernel of the map from $P$ to  $gr^{1.1}(J_{\infty}(A))$, by sending $E_{i,j}$ to its coset. In Step 1.2 we define $$H_{1.2}:=\mathbb{C}[E_{i,j}|1 \leq i<j \leq n]/\mathbb{T}_{1.2}, $$  where $\mathbb{T}_{1.2}$ is the kernel of the map from $P$ to $gr^{1.2}(J_{\infty}(H_{1.1}))$.
Next we move to Step 2.1, etc. until we construct $gr^{n-2,n-1}(J_\infty(H_{n-2,n-2}))$ and $\mathbb{T}_{n-2,n-1}$. As before we can prove that 
$$\mathcal{I} \subset \mathbb{T}_{n-2,n-1}.$$

Then using exactly the same type of arguments as in Section 3.2,  we can prove the following result.
\begin{prop} \label{sec-ineq}
We have $$ HS_{q}(J_{\infty}(A))\leq HS_{q}(J_{\infty}(H)).$$
\end{prop}

Next we compute Hilbert series of the jet algebras $J_\infty(B)$ and $J_\infty(H)$.
\begin{prop}\label{jetchar}
The Hilbert series of $J_{\infty}(B)$ is  \begin{align}\label{iden0}\sum_{{\bf m} \in \mathbb{N}_{\geq 0}^{n(n-1)/2}} \frac{{\displaystyle q^{B({\bf m})}}}{\displaystyle\prod_{1\leq i<j\leq n }(q)_{n_{i,j}}},\end{align} and the Hilbert series of $J_{\infty}(H)$ is  \begin{align}\label{iden}\sum_{{\bf m} \in \mathbb{N}_{\geq 0}^{n(n-1)/2}} \frac{{\displaystyle q^{B'({\bf m})}}}{\displaystyle\prod_{1\leq i<j\leq n }(q)_{n_{i,j}}},\end{align}
where quadratic forms $B({\bf m})$ and $B'({\bf m})$ are as in Theorem \ref{thm1}  (not to be confused with the algebra $B$).
\end{prop}

\begin{proof}
First, we consider an integral lattice $L$ of rank $\frac{n(n-1)}{2}$ with a basis:  \[\alpha_{1,2},\ldots,\alpha_{1,n},\ldots,\alpha_{n-1,n}.\]  The symmetric bilinear form of this lattice is defined as:
$(\alpha_{i_{1},j_{1}},\alpha_{i_{2},j_{2}})=1$ when \begin{itemize}
    \item $1\leq i_1 < j_1\leq n,\,1\leq i_2 < j_2\leq n,\,1\leq i_1 < i_2\leq n,\,1\leq j_1 < j_2\leq n ,\,j_{1}>i_{2}+1,$
    \item $1\leq i_1 < j_1\leq n,\,1\leq i_2 < j_2\leq n,\,1\leq i_{1}=i_{2}\leq n ,\, 1\leq j_{1}< j_{2}\leq n,$
    \item $1\leq i_1 < j_1\leq n,\,1\leq i_2 < j_2\leq n,\,1\leq j_{1}=j_{2}\leq n ,\, 1\leq i_1<i_2\leq n,$
    \item $i_{1}+1=j_{1}, \,1\leq  i_{2}\leq i_{1}<i_{1}+1<j_{2}\leq n,$
\end{itemize}
 $(\alpha_{i,j},\alpha_{i,j})=2$ for all $\alpha_{i,j}$, and $(\alpha_{i_{1},j_{1}},\alpha_{i_{2},j_{2}})=0$ otherwise. According to \cite[Theorem 5.13]{li2020some}, we have $J_{\infty}(B)\cong gr^{F}(W_{L}).$ It is known that the character of  $W_{L} \subset V_L$ is given by ($\ref{iden0}$) \cite{milas2012lattice,li2020some,kawasetsu}. Thus, we proved the first identity.
 The second identity for $HS_q(J_{\infty}(H))$ follows along the same lines.
\end{proof}

Equipped with these explicit formulas, we reduced the problem of classical freeness of $W(\Lambda_0)$ to a statement about characters and Hilbert series.





\section{Quantum dilogarithm}




The quantum dilogarithm is an important infinite series/product defined as $$\phi(x):=\prod_{i \geq 0} (1-q^i x).$$
Using the binomial $q$-series identity we also have
$$\phi(x)=\sum_{n \geq 0} (-1)^n \frac{q^{\frac12 n(n-1)} x^n}{(q)_n}.$$
and therefore 
\begin{equation} \label{Euler}
\phi(-q^{\frac{j}{2}}x)=\sum_{n \geq 0} \frac{q^{\frac12 n^2+\frac{j-1}{2}n} x^n}{(q)_n}.
\end{equation}

Let $x$ and $y$ be non-commutative variables such that
$$xy=q yx.$$ Then the quantum dilogarithm satisfies an important
pentagon identity of Faddeev and Kashaev \cite{faddeev1994quantum},
\begin{equation} \label{pent}
\phi(y) \phi(x)=\phi(x) \phi(-yx) \phi(y).
\end{equation}
This identity is widely used in the derivation of wall-crossing invariants.
\subsection{Warm up: Proof of (\ref{sl3}) using quantum dilogarithm}

We first observe that the two sides are equal if and only if ${\rm Coeff}_{x^m y^n}$ agree.
Comparing coefficients on both sides leads to
$$\frac{q^{m^2+n^2-mn}}{(q)_m (q)_n}=\sum_{n_1+n_2=m \atop n_2+n_3=n} \frac{q^{n_1^2+n_2^2+n_3^2+n_1 n_2+n_2 n_3}}{(q)_{n_1} (q)_{n_2} (q)_{n_3}},$$
where, after letting $n_1=m-n_2$ and $n_3=n-n_2$, the right hand-side can be rewritten as
\begin{equation} \label{sl3-coeff}
\sum_{n_2 \geq 0} \frac{q^{(m-n_2)^2+(n-n_2)^2+n_2^2+(m-n_2)n_2+(n-n_2)n_2}}{(q)_{m-n_2}(q)_{n-n_2}(q)_{n_2}}.
\end{equation}
After simplifying exponents on both sides we end up with an equivalent identity
\begin{equation} \label{sl3-alt}
\frac{1}{(q)_m(q)_n}=\sum_{n_2 \geq 0} \frac{q^{(n-n_2)(m-n_2)}} {(q)_{m-n_2}(q)_{n-n_2}(q)_{n_2}}.
\end{equation}
This famous identity appeared in a variety of situation and there are several different proofs in the literature
\cite{zagier}, \cite{lee} (credited to Zwegers), \cite{faddeev1994quantum}, etc.

Using quantum pentagon identity (\ref{pent1}) in a slightly different form
\begin{equation} \label{pent1}
\phi(-q^{1/2} y) \phi(-q^{1/2} x)=\phi(-q^{1/2} x) \phi(- q yx) \phi(-q^{1/2} y)
\end{equation}
together with Euler's expansion we get
$$\sum_{m,n \geq 0} \frac{q^{\frac12 m^2+\frac12 n^2} y^m x^n}{(q)_m (q)_n}=\sum_{n_1,n_2,n_3 \geq 0} \frac{q^{\frac12 n_1^2+\frac12 n_2(n_2+1)+\frac12 n_3^2} x^{n_1} (yx)^{n_2} y^{n_3}}{(q)_{n_1} (q)_{n_2} (q)_{n_3}}.$$
Using easy-to-verify identities
$$y^m x^n = x^n y^m q^{-mn}, (yx)^{n_2}=x^{n_2} y^{n_2} q^{-\frac12 n_2(n_2+1)}$$
we can write 
\begin{equation} \label{qd1}
\sum_{m,n \geq 0} \frac{q^{\frac12 m^2+\frac12 n^2-mn} x^n y^m }{(q)_m (q)_n}=\sum_{n_1,n_2,n_3 \geq 0} \frac{q^{\frac12 n_1^2+\frac12 n_3^2} x^{n_1+n_2}  y^{n_2+n_3}}{(q)_{n_1} (q)_{n_2} (q)_{n_3}}.
\end{equation}
Extracting the term next to $x^n y^m$ now gives (after letting $n=n_1+n_2$ and $m=n_2+n_3$)
$$\frac{q^{\frac12 m^2+\frac12 n^2-mn}  }{(q)_m (q)_n}=\sum_{n_2 \geq 0} \frac{q^{\frac12 (n-n_2)^2+\frac12 (m-n_2)^2 }}{(q)_{n-n_2} (q)_{m-n_2} (q)_{n_2}}$$
which is clearly equivalent to identity (\ref{sl3-alt}).
\qed

In summary, the pentagonal identity for the quantum dilogarithm (\ref{qd1}) is equivalent (according to our definition) to the character 
identity for the level one principal subspace of $\fsl_{3}$ (\ref{sl3}).

\subsection{General case}


Let $x_i$, $1 \leq i \leq n-1$, be non-commutative variables. Assume that
\begin{equation}\label{commu1}
x_i x_{i+1}=q x_{i+1}x_i,
\end{equation} and other pairs commute. Then

\begin{align} \phi \left(-\frac{q}{2}x_{n-1} \right) \cdots \phi \left(-\frac{q}{2}x_1\right)=\sum_{k_1,\cdots,k_{n-1} \geq 0} \frac{q^{\frac{k_1^2}{2}+ \cdots + \frac{k_{n-1}^2}{2}-k_1 k_2 -\cdots-k_{n-2} k_{n-1}} x_1^{k_1} \cdots x_{n-1}^{k_{n-1}}}{(q)_{k_1} \cdots (q)_{k_{n-1}}}. \end{align}
By definition we have two useful formulas:
\begin{align}\begin{split}\phi(-q^{\frac{j-i}{2}}x_{j-1}x_{j-2}\cdots x_{i})&=\displaystyle \sum_{m\geq 0}(-1)^{m}\frac{q^{\frac{m(m-1)}{2}}(-q^{\frac{j-i}{2}}x_{j-1}\cdots x_{i})^{m}}{(q)_{m}} \\&= \displaystyle \sum_{m\geq 0}\frac{q^{\frac{m^{2}}{2}}q^{\frac{m(j-i)}{2}-\frac{m}{2}}(x_{j-1}\cdots x_{i})^{m}}{(q)_{m}},\end{split}\end{align}

\begin{align}(x_{j-1}\cdots x_{i})^{m}=q^{-(j-i-1)\frac{m(m+1)}{2}}x_{i}^{m}x_{i+1}^{m}\cdots x_{j-1}^{m}.\end{align}

These two formulas combined together give us \begin{align}\label{formula1}\phi(-q^{{\frac{j-i}{2}}}x_{j-1}x_{j-2}\cdots x_{i})=\displaystyle \sum_{m\geq 0}\frac{q^{\frac{(2-(j-i))m^{2}}{2}}x_{i}^{m}x_{i+1}^{m}\cdots x_{j-1}^{m}}{(q)_{m}}.\end{align}

Now for $\phi(-q^{\frac{1}{2}}x_{n-1}) \cdots \phi(-q^{\frac{1}{2}}x_1)$ and formula (\ref{pent1}) we commute factors such that $\phi(-q^{\frac{1}{2}}x_i)$ is in front of $\phi(-q^{\frac{1}{2}}x_{i+1}).$ Therefore, we get the quantum dilogarithmic identity (in this particular order!): \begin{align}\label{qdi} \begin{split} &\phi(-q^{\frac{1}{2}}x_{n-1}) \cdots \phi(-q^{\frac{1}{2}}x_1)\\  =&\phi(-q^{\frac{1}{2}} x_1)\\ & \cdot  \phi(-q x_2x_{1})\phi(-q^{\frac{1}{2}} x_2)  \\& \mathrel{\makebox[\widthof{=}]{\vdots}} \\ & \cdot \phi(-q^{\frac{n-2}{2}}x_{n-2}\cdots x_1)\phi(-q^{\frac{n-3}{2}}x_{n-2}\cdots x_2)\cdots \phi(-q^{\frac{1}{2}}x_{n-2})\\ & \cdot \phi(-q^{\frac{n-1}{2}}x_{n-1}\cdots x_1)\phi(-q^{\frac{n-2}{2}}x_{n-1}\cdots x_2)\cdots \phi(-qx_{n-1} x_{n-2}) \phi(-q^{\frac{1}{2}}x_{n-1}).\end{split} \end{align}
We expand for each $j>i$:
\begin{align}\phi(-q^{\frac{j-i}{2}}x_{j-1}x_{j-1}\cdots x_{i})=\displaystyle \sum_{m_{i,j}\geq 0} (-1)^{m_{i,j}}\frac{q^{\frac{m_{i,j}(m_{i,j}-1)}{2}}(-q^{\frac{j-i}{2}}x_{j-1}x_{j-2}\cdots x_{i})^{m_{i,j}}}{(q)_{m_{i,j}}}. \end{align} 
This allows us to rewrite dilogarithmic identity (\ref{qdi}) as an identity in (\ref{generic2}):


\begin{align}\label{qdi1} \sum_{{\bf m} \in \mathbb{N}_{\geq 0}^{n(n-1)/2}} \frac{ q^{C({\bf m})}\cdot F} {\displaystyle\prod_{1\leq i<j\leq n }(q)_{m_{i,j}}}=\sum_{k_1,\cdots,k_{n-1} \geq 0} \frac{q^{\frac{k_1^2}{2}+ \cdots + \frac{k_{n-1}^2}{2}-k_1 k_2 -\cdots-k_{n-2} k_{n-1}} x_1^{k_1} \cdots x_{n-1}^{k_{n-1}}}{(q)_{k_1} \cdots (q)_{k_{n-1}}},\end{align} where \begin{align*} C({\bf m}) = \sum_{\substack{m_{i,j}\geq 0\\1\leq i<j\leq n }}\frac{(2-(j-i))m_{i,j}^{2}}{2}\end{align*} and \begin{align*}
    F=(x_{1}^{m_{1,2}}) (x_{1}^{m_{1,3}}x_{2}^{m_{1,3}+m_{2,3}})\cdots(x_{1}^{m_{1,n}}\cdots x_{n-1}^{m_{1,n}}x_{2}^{m_{2,n}}\cdots x_{n-1}^{m_{2,n}}\cdots x_{n-1}^{m_{n-1,n}}).
\end{align*}  An application of (\ref{commu1}) to $F$, we get $ F=q^{E({\bf m})}x_{1}^{\lambda_{1}}\cdots x_{n}^{\lambda_{n}},$  where $E({\bf m})$ is some quadratic form and \[\lambda_i=\displaystyle \sum _{ 1 \leq s\leq i \atop i < \ell \leq n}m_{s,\ell}.\] For identities (\ref{main}) and (\ref{qdi1}), we let


$$m_{i,i+1}:=k_{i}-\displaystyle \sum _{ 1 \leq s< i \atop i+1 < \ell \leq n}m_{s,\ell}-\sum _{ s=i \atop i+1 < \ell \leq n}m_{s,\ell}-\sum _{ 1 \leq s< i \atop  \ell= i+1  }m_{s,\ell}.$$


\begin{thm}\label{equi}
The identity (\ref{main}) and the quantum dilogarithm identity (\ref{qdi}) are equivalent.
\end{thm}

\begin{proof}

We will prove the statement by induction. We already proved that it holds for $A_{2}$ in the previous section (base case). Assume that quantum dilogarithmic identity (\ref{qdi}) and identity (\ref{main}) are equivalent for $A_{n-1}$.

Now we prove the equivalence for $A_{n}$. In order to prove that, it suffices to show
(with $m_{i,i+1}$ as above) that  \begin{equation} \label{coef} C({\bf m})+E({\bf m})-(\frac{k_1^2}{2}+ \cdots + \frac{k_n^2}{2}-k_1 k_2 -..-k_{n-1} k_n)=B({\bf m})-\frac12{\mathbf{k}A\mathbf{k}^\top}.\end{equation} Letting $m_{i,n+1}=0$ where $1\leq i \leq n,$ in the above equation gives our induction hypothesis. Therefore we need to show that both sides of (\ref{coef}) have the same terms which involve $m_{i,n+1}$ ($1\leq i \leq n$). After expanding both sides of (\ref{coef}), it is easy to see the terms of the form $k_{i}k_{j}$ $(j> i+1)$ are absent on both sides, and the coefficients of the terms of the form $k_{i}k_{i+1}$ equal $1$ on both sides (so they cancel out).  
Similarly, terms of the form $k_{i}^{2}$ also cancel out.


We are left to analyze $6$ possible types of terms involving $m_{i,n+1}$:
\begin{itemize}
    \item[{\bf Type \RomanNumeralCaps 1:}] $m_{i_{2},j_{2}+1}m_{i_{1},n+1},$ where $j_{2}+1\leq i_{1}-1$.

\begin{tikzpicture}
\draw[gray, thick] (-15,-2) -- (-12,-2);
\draw[gray, thick] (-10,-1) -- (-7,-1);
\draw[gray, thick,dashed] (-10,-1) -- (-11,-1);
\filldraw[black] (-10,-1) circle (2pt) node[anchor=south] {$i_{1}$};
\filldraw[black] (-7,-1) circle (2pt) node[anchor=south] {$n+1$};
\filldraw[black] (-15,-2) circle (2pt) node[anchor=south] {$i_{2}$};
\filldraw[black] (-12,-2) circle (2pt) node[anchor=south] {$j_{2}+1$};
\filldraw[black] (-11,-1) circle (2pt) node[anchor=south] {$i_{1}-1$};
\end{tikzpicture}

\item[{\bf Type \RomanNumeralCaps 2:}] $m_{i_{2},j_{2}+1}m_{i_{1},n+1},$ where $i_{2}<i_{1}\leq j_{2}+1<n+1$.

\begin{tikzpicture}
\draw[gray, thick] (-15,-2) -- (-10,-2);
\draw[gray, thick] (-12,-1) -- (-7,-1);
\draw[gray, thick,dashed] (-10,-1) -- (-11,-1);
\filldraw[black] (-12,-1) circle (2pt) node[anchor=south] {$i_{1}$};
\filldraw[black] (-7,-1) circle (2pt) node[anchor=south] {$n+1$};
\filldraw[black] (-15,-2) circle (2pt) node[anchor=south] {$i_{2}$};
\filldraw[black] (-10,-2) circle (2pt) node[anchor=south] {$j_{2}+1$};

\end{tikzpicture}

\item[{\bf Type \RomanNumeralCaps 3:}] $m_{i_{2},j_{2}+1}m_{i_{1},n+1},$ where $i_{1}\leq i_{2}<j_{2}+1<n+1$, or $i_{1}< i_{2}<j_{2}+1\leq n+1$.

\begin{tikzpicture}
\draw[gray, thick] (-12,-2) -- (-9,-2);
\draw[gray, thick] (-14,-1) -- (-7,-1);
\draw[gray, thick,dashed] (-10,-1) -- (-11,-1);
\filldraw[black] (-14,-1) circle (2pt) node[anchor=south] {$i_{1}$};
\filldraw[black] (-7,-1) circle (2pt) node[anchor=south] {$n+1$};
\filldraw[black] (-12,-2) circle (2pt) node[anchor=south] {$i_{2}$};
\filldraw[black] (-9,-2) circle (2pt) node[anchor=south] {$j_{2}+1$};

\end{tikzpicture}

\item[{\bf Type \RomanNumeralCaps 4:}] $k_{i}m_{i_{1},n+1},$ where $i_{1}\geq i+1$.

\begin{tikzpicture}
\draw[gray, thick] (-13,-2) -- (-13,-2);
\draw[gray, thick] (-10,-1) -- (-6,-1);
\draw[gray, thick,dashed] (-10,-1) -- (-11,-1);
\filldraw[black] (-10,-1) circle (2pt) node[anchor=south] {$i_{1}$};
\filldraw[black] (-6,-1) circle (2pt) node[anchor=south] {$n+1$};
\filldraw[black] (-13,-2) circle (2pt) node[anchor=south] {$i$};
\filldraw[black] (-11,-1) circle (2pt) node[anchor=south] {$i_{1}-1$};

\end{tikzpicture}

\item[{\bf Type \RomanNumeralCaps 5:}] $k_{i}m_{i_{1},n+1},$ where $n\geq i\geq i_{1}$.

\begin{tikzpicture}
\draw[gray, thick] (-10,-2) -- (-10,-2);
\draw[gray, thick] (-13,-1) -- (-9,-1);

\filldraw[black] (-13,-1) circle (2pt) node[anchor=south] {$i_{1}$};
\filldraw[black]  (-9,-1) circle (2pt) node[anchor=south] {$n+1$};
\filldraw[black] (-10,-2) circle (2pt) node[anchor=south] {$i$};

\end{tikzpicture}
\item[{\bf Type \RomanNumeralCaps 6:}] $m_{i,n+1}^{2},$ where $n\geq i$.
\end{itemize}
\vspace{3mm}

Now we compare terms of each type  on the  left- and right-hand side of (\ref{coef}).
 We only provide a few details for the sake of brevity.
First, on the left-hand side of (\ref{coef}), straightforward computations with powers of the $q$-series give the following coefficients:  for terms of Type {\RomanNumeralCaps 1} and {\RomanNumeralCaps 4} coefficients are zero as they are absent from the formula. Similarly, for all Type {\RomanNumeralCaps 3} terms coefficients are also zero. For Type {\RomanNumeralCaps 5} and {\RomanNumeralCaps 6} the coefficients are $-1$ and $1$, respectively. For Type  {\RomanNumeralCaps 2}, we are interested in $q$-powers of $m_{i_{2},j_{2}+1}m_{i_{1},n+1}$. Their contribution comes from two sources:

\item[{\bf a}.] $q^{E({\bf m})}$, obtained from rewriting the product of non-commutative variables 
\begin{align*} (\prod_{s=i_{2}}^{j_{2}}x_{s})^{m_{i_{2},j_{2}+1}}\cdots (\prod_{s=i_{1}}^{n}x_{s})^{m_{i_{1},n+1}}\cdots, \end{align*}
in the form $x_{1}^{\lambda_{1}} \cdots x_{n}^{\lambda_{n}}$ This part contributes with $q^{-(j_{2}-i_{1})m_{i_{1},n+1}m_{i_{2},j_{2}+1}},$
        \item[{\bf b.}] $q^{C({\bf m})}$, where we have
        \begin{align*} \prod_{s=i_{1}}^{j_{2}}q^{\frac{m^{2}_{s,s+1}}{2}}=q^{(j_{2}-i_{1}+1)m_{i_{2},j_{2}+1}m_{i_{1},n+1}+ \cdots } . \end{align*}
Therefore, the overall coefficient of {Type \RomanNumeralCaps 2} terms (in the $q$-exponent) is $(j_{2}-i_{1}+1) -(j_{2}-i_{1})=1$.

Now we compute coefficients on the right-hand side of (\ref{coef}) for the same six types.
Again, Type  {\RomanNumeralCaps 1} terms are absent so their coefficients are zero.
We list all possible ways, in which we get contribution of {Type \RomanNumeralCaps 2} terms:
    \begin{itemize}
        \item[{\bf a.}] \begin{align*}\sum_{s=i_{1}}^{j_{2}}m^{2}_{s,s+1}=2(j_{2}-i_{1}+1)m_{i_{2},j_{2}+1}m_{i_{1},n+1}+\cdots, \end{align*}
        \item[{\bf b.}] \begin{align*}(\sum_{s=i_{1}}^{j_{2}}m_{s,s+1})m_{i_{1},n+1}=-(j_{2}-i_{1}+1)m_{i_{2},j_{2}+1}m_{i_{1},n+1}+\cdots,\end{align*}
        \item[{\bf c.}] \begin{align*}(\sum_{s=i_{1}}^{j_{2}}m_{s,s+1})m_{i_{2},j_{2}+1}=-(j_{2}-i_{1}+1)m_{i_{2},j_{2}+1}m_{i_{1},n+1}+\cdots,\end{align*}
        \item[{\bf d.}] \[m_{i_{2},j_{2}+1}m_{i_{1},n+1}.\]\vspace{0,2cm}
 \end{itemize}
Adding these up gives the coefficient of {Type \RomanNumeralCaps 2} to be $1.$
Along the same lines  Type {\RomanNumeralCaps 3} and {\RomanNumeralCaps 4} coefficients are zero, and Type {\RomanNumeralCaps 5} and {\RomanNumeralCaps 6} coefficients are $-1$ and $1$, respectively.

From above argument, we see that both sides of (\ref{coef}) have the same terms which involve $m_{i,n+1}$ ($1\leq i \leq n$).  Therefore, two identities are equivalent by induction.
\end{proof}

Next result is needed in Section 5 and can be proven as above. Let $B'({\bf m})$ be as in Theorem 1.1 in the introduction.
\begin{prop}\label{Rem Co}
For the case of $A_{n}$, the coefficients of six types of terms in $B'({\bf m})-\frac12{\mathbf{k} A\mathbf{k}^\top}$ are:
\begin{table}[H]
  \begin{center}
    \label{tab:table1}
    \begin{tabular}{l|c|r} 
      \textbf{Type} & \textbf{Coefficient} & \textbf{Condition} \\
      \hline
      {\RomanNumeralCaps 1} & 0  & \\
      \RomanNumeralCaps 2 & $2(j_{2}-i_{1})+1$ & \\
      \RomanNumeralCaps 3 & $2(j_{2}-i_{2})$ & \\
      \RomanNumeralCaps 4 & 0 & \\
      \RomanNumeralCaps 5 & -1 & $i=i_1$, $i=n$ \\
      \RomanNumeralCaps 5 & -2 & $i_1 < i < n$ \\
      \RomanNumeralCaps 6 &  n-i & \\
    \end{tabular}
  \end{center}
\end{table}
\end{prop}
\begin{cor} Let $W(\Lambda_{0})$ be the principal subspace of the affine vertex algebra $L_{1}(\fsl_{n})$ as in Section 3.
Then the Hilbert series of $J_{\infty}(R_{W(\Lambda_{0})})$ is given by \[\sum_{{\bf m} \in \mathbb{N}_{\geq 0}^{n(n-1)/2}} \frac{{ q^{B({\bf m})}} }{\prod_{1\leq i<j\leq n }(q)_{n_{i,j}}}.\]
\end{cor}
\begin{proof}
According to Proposition \ref{filtra}, Proposition \ref{jetchar} and Theorem \ref{equi}, we know that (recall $A=R_{W(\Lambda_0)}$) \[HS_{q}(J_{\infty}(R_{W(\Lambda_{0})}))\leq {\rm ch}[W(\Lambda_{0})](q). \]  Since there exists a grading preserving surjective homomorphism $$\psi: J_{\infty}(R_{W(\Lambda_{0})})\twoheadrightarrow gr^{F}(W(\Lambda_{0})),$$ we also have \[HS_{q}(J_{\infty}(R_{W(\Lambda_{0})}))\geq {\rm ch}[W(\Lambda_{0}))](q)\] and the assertion follows.
\end{proof}

Finally,  we can conclude:
\begin{cor} \label{main-cor}
The principal subspace $W(\Lambda_0)$ of $L_1(\fsl_n)$ is classically free.
\end{cor}



\section{$q$-identities from quiver representations}

In this part we discuss identity (\ref{main4}) from the perspective of quiver representations following \cite{rimanyi2018partition}. We will use notations and some formulas from \cite{rimanyi2018partition}, including several basic facts about quiver representations.
Indecomposable representations of the quiver $Q$ of type $A_{n-1}$ are in one-to-one correspondence with positive roots of $A_{n-1}$ enumerated by the segments: $[i,j]$, $1 \leq i \leq j \leq n-1$.
\begin{thm}\cite[Corollary 1.5]{rimanyi2018partition} \label{rr} Let $Q$ be a quiver of type $A_{n-1}$. Then for every
${\bf k}=(k_1,...,k_{n-1}) \in \mathbb{N}^{n-1}_{\geq 0}$ we have an identity
\begin{align}\label{Rimany} \frac{1}{\displaystyle\prod_{i=1}^{n-1} (q)_{k_i}}=\sum_{\eta} q^{{\rm codim}(\eta)} \frac{1}{\displaystyle \prod_{1\leq i\leq j\leq n-1}(q)_{m_{[i,j]}(\eta)}}\end{align}
where summation is over all finite-dimensional representations $\eta$ (up to equivalence) of $Q$ such that
$\overline{{\rm dim}}(\eta)={\bf k}$, ${\rm codim}(\eta)$ is the co-dimension of a certain orbit (see below for an explicit formula), and $m_{[i,j]}(\eta)$ indicates the multiplicity of $[i,j]$ in $\eta$.
\end{thm}


Recall another result from \cite{rimanyi2018partition} for codimensions:
\begin{lem}
$${\rm codim}(\eta)=\sum_{[I,J] \in {\rm ConditionStrands}} m_{I} \cdot m_{J},$$
where
$${\rm ConditionStrands}=\{[I,J]: [I,J] \ {\rm satisfies \ conditions } \ (1), (2) \ {\rm or} \ (3) \},$$
where strand is a pair $I=[a,b]$, $1 \leq a \leq b \leq n-1$ (corresponding to indecomposable rep of $Q$) and


(1) $I=[w,x-1], J=[x,z]$, $w < x \leq z,$ e.g.,

\begin{tikzpicture}
\draw[gray, dashed] (-9,-2) -- (-6,-2);
\draw[gray, dashed] (-13,-1) -- (-10,-1);
\filldraw[black] (-13,-1) circle (2pt) node[anchor=south] {$w$};
\filldraw[black] (-10,-1) circle (2pt) node[anchor=south] {$x-1$};
\filldraw[black] (-9,-2) circle (2pt) node[anchor=south] {$x$};
\filldraw[black] (-6,-2) circle (2pt) node[anchor=south] {$z$};

\end{tikzpicture}

(2) $I=[w,y], J=[x,z]$, $w < x \leq y <z$ and the arrows $a_{x-1}$ and $a_{y}$ point in the same direction, e.g.,

\begin{tikzpicture}\usetikzlibrary{decorations.markings}
\draw[gray, dashed] (-11,-2) -- (-8,-2);
\draw[gray, dashed] (-13,-1) -- (-10,-1);
\filldraw[black] (-13,-1) circle (2pt) node[anchor=south] {$w$};
\filldraw[black] (-10,-1) circle (2pt) node[anchor=south] {$y$};
\filldraw[black] (-11,-2) circle (2pt) node[anchor=south] {$x$};
\filldraw[black] (-8,-2) circle (2pt) node[anchor=south] {$z$};
\filldraw[black] (-12,-1) circle (2pt);
\filldraw[black] (-11,-1) circle (2pt);
\filldraw[black] (-10,-2) circle (2pt);
\filldraw[black] (-9,-2) circle (2pt);
\begin{scope}[thick,decoration={
    markings,
    mark=at position 0.5 with {\arrow{>}}}
    ]
    \draw[postaction={decorate}] (-12,-1)--(-11,-1);
    \draw[postaction={decorate}] (-10,-2)--(-9,-2);
\end{scope}
\end{tikzpicture}

(3) $I=[x,y], J=[w,z]$, $w < x \leq  y <z$ and the arrows $a_{x-1}$ and $a_{y}$ point in different directions, e.g.,

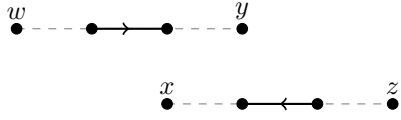
\begin{tikzpicture}
\draw[gray, dashed] (-11,-2) -- (-8,-2);
\draw[gray, dashed] (-13,-1) -- (-10,-1);
\filldraw[black] (-13,-1) circle (2pt) node[anchor=south] {$w$};
\filldraw[black] (-10,-1) circle (2pt) node[anchor=south] {$y$};
\filldraw[black] (-11,-2) circle (2pt) node[anchor=south] {$x$};
\filldraw[black] (-8,-2) circle (2pt) node[anchor=south] {$z$};
\filldraw[black] (-12,-1) circle (2pt);
\filldraw[black] (-11,-1) circle (2pt);
\filldraw[black] (-10,-2) circle (2pt);
\filldraw[black] (-9,-2) circle (2pt);

\begin{scope}[thick,decoration={
    markings,
    mark=at position 0.5 with {\arrow{>}}}
    ]
    \draw[postaction={decorate}] (-12,-1)--(-11,-1);
    \draw[postaction={decorate}] (-9,-2)--(-10,-2);
\end{scope}

\end{tikzpicture}

\end{lem}

\vspace{5mm}
For our purpose, we rewrite (\ref{Rimany}) as a family of identities in the shape of (\ref{generic2}):



\begin{align}\label{Ram} &  \displaystyle \sum_{{\bf k}=(k_1,...,k_{n-1}) \in \mathbb{N}^{n-1}_{\geq 0}}\frac{ {\bf x^{k}}}{\prod_{i=1}^{n-1} (q)_{k_i}}  =\sum_{\eta } \frac{q^{{\rm codim }(\eta)}{\bf x^{\overline{dim}(\eta)} }}   {\displaystyle \prod_{1\leq i\leq j\leq n-1}(q)_{m_{[i,j]}(\eta)}},
\end{align}
and the summation is over all finite-dimensional representation $\eta$ of $A_{n-1}$ (up to isomorphism).
\vspace{5mm}
\begin{ex}[$\fsl_{3}$]
Representation $\eta$ of the $A_2$ quiver of
type $\overline{{\rm dim}}(\eta)=(m,n)$ over complex numbers is given by ${\mathbb{C}^m \atop \circ } \rightarrow {\mathbb{C}^n \atop \circ}$.
Indecomposable representations are given by $[1,1]:={\mathbb{C} \atop \circ } \rightarrow {0 \atop \circ}$,
 $[2,2]:={0 \atop \circ } \rightarrow {\mathbb{C} \atop \circ}$, and $[1,2]={\mathbb{C} \atop \circ } \rightarrow {\mathbb{C} \atop \circ}$.
Condition (1) is $[1,1],[2,2]$ and there are no pairs with conditions two and three. Therefore
we have
$${\rm codim}(\eta)=m_{[1,1]}(\eta) \cdot m_{[2,2]}(\eta),$$
and we are summing over all representations $\eta$ with $\overline{dim}(\eta)={\bf k}=(m,n)$. That means

$$\frac{1}{(q)_m(q)_n}=\sum_{{m_{[1,1]},m_{[2,2]},m_{[1,2]} \geq 0 \atop m_{[1,1]}+m_{[1,2]}=m} \atop m_{[2,2]}+m_{[1,2]}=n }q^{m_{[1,1]} \cdot m_{[2,2]}}\frac{1}{(q)_{m_{[1,1]}} (q)_{m_{[2,2]}} (q)_{m_{[1,2]}}},$$
which is precisely formula (8) (we only have to rewrite $m-n_2=n_1$ and $n-n_2=n_3$, and use
that summation variables are $m_{[1,1]}=n_1$, $m_{[1,2]}=n_2$ and $m_{[2,2]}=n_3$).


\end{ex}


Then we can prove the following:
\begin{thm}
The identities (\ref{Ram}) and (\ref{main4}) are equivalent for $n\geq 2$.
\end{thm}

\begin{proof}We have proved the equivalence for $A_{2}$. We assume the edges of the $A_{n}$ quiver are all oriented from left to  right, i.e,  \begin{tikzpicture}

\draw[gray, dashed] (-13,-1) -- (-10,-1);
\filldraw[black] (-13,-1) circle (2pt);
\filldraw[black] (-10,-1) circle (2pt);
\filldraw[black] (-12,-1) circle (2pt);
\filldraw[black] (-11,-1) circle (2pt);

\begin{scope}[thick,decoration={
    markings,
    mark=at position 0.5 with {\arrow{>}}}
    ]
    \draw[postaction={decorate}] (-12,-1)--(-11,-1);
\end{scope}

\end{tikzpicture}. Because of the orientation, there is no pair with condition (3) appearing in the formula for ${\rm codim}(\eta)$. 

We will prove that these two $q$-series identities are equivalent using the induction on $n$. Assume that (\ref{Ram}) and (\ref{main4}) are equivalent for $A_{n-1}.$ For the case of $A_{n}$, by identifying $m_{[i,j]}$ with $m_{i,j+1}$, and letting  $m_{i,i+1}$ (i.e. $m_{[i,i]}$)  be

$$k_{i}-\displaystyle \sum _{ 1 \leq s< i \atop i+1 < \ell \leq n+1}m_{s,\ell}-\sum _{ s=i \atop i+1 < \ell \leq n+1}m_{s,\ell}-\sum _{ 1 \leq s< i \atop  \ell= i+1  }m_{s,\ell},$$
 it is enough to show that  
\begin{align}\label{new} {\rm codim}(\eta)=B'({\bf m})-\frac12{\mathbf{k}A\mathbf{k}^\top}.\end{align} It is clear that we have the same terms involving $k_{i}k_{j}$ $(1\leq i\leq j\leq n)$ on both sides of (\ref{new}). Note that we also have 6 types of terms that involve  \[m_{[i_{2},j_{2}]}m_{[i_{1},n]}\quad (m_{i_{2},j_{2}+1}m_{i_{1},n+1}).\]  Then we compute the coefficients of these terms for ${\rm codim}(\eta)$.

\begin{itemize}
    \item Terms in ${\rm codim}(\eta)$ that involve {Type \RomanNumeralCaps 1} term are the following: \begin{itemize}
        \item[{\bf a.}] $m_{[i_{2},i_{1}-1]}m_{[i_{1},i_{1}]}=-m_{[i_{2},i_{1}-1]}m_{[i_{1},n]}+\cdots,$
        \item[{\bf b.}] $m_{[i_{1}-1,i_{1}-1]}m_{[i_{1},n]}=-m_{[i_{2},i_{1}-1]}m_{[i_{1},n]}+\cdots,$
        \item[{\bf c.}] $m_{[i_{1}-1,i_{1}-1]}m_{[i_{1},i_{1}]}=m_{[i_{2},i_{1}-1]}m_{[i_{1},n]}+\cdots,$
        \item[{\bf d.}] $m_{[i_{2},i_{1}-1]}m_{[i_{1},n]}.$
    \end{itemize}
    Therefore, the coefficient of each {Type \RomanNumeralCaps 1} term is 0.
    \item We list all terms that involve {Type \RomanNumeralCaps 2} term as following:


     \begin{itemize}
    \item[{\bf a.}] \begin{align*} &m_{[i_{1}-1,i_{1}-1]}m_{[i_{1},i_{1}]}+\sum_{s=i_{1}}^{j_{2}-1}m_{[s,s]}m_{[s+1,s+1]} +m_{[j_{2},j_{2}]}m_{[j_{2}+1,j_{2}+1]}\\ &=(2(j_{2}-i_{1})+2)m_{[i_{2},j_{2}]}m_{[i_{1},n]}+\cdots,\end{align*}
     \item[{\bf b.}] \[m_{[i_{1}-1,i_{1}-1]}m_{[i_{1},n]}=-m_{[i_{2},j_{2}]}m_{[i_{1},n]}+\cdots,\]

         \[m_{[i_{2},j_{2}]}m_{[j_{2}+1,j_{2}+1]}=-m_{[i_{2},j_{2}]}m_{[i_{1},n]}+\cdots,\]

     \item[{\bf c.}] \[m_{[i_{2},j_{2}]}m_{[i_{1},n]}.\]
    \end{itemize}
    \vspace{0.1cm}
   Therefore, the coefficient of each {Type \RomanNumeralCaps 2} term is $2(j_{2}-i_{1})+1.$

\item For {Type \RomanNumeralCaps 3} terms, we list all possible contributions from ${\rm codim}(\eta)$:

\begin{itemize}
    \item[{\bf a.}] \begin{align*}
        & m_{[i_{2}-1,i_{2}-1]}m_{[i_{2},i_{2}]}+\sum_{s=i_{2}}^{j_{2}-1}m_{[s,s]}m_{[s+1,s+1]}+m_{[j_{2},j_{2}]}m_{[j_{2}+1,j_{2}+1]}\\&=(2(j_{2}-i_{2})+2)m_{[i_{2},j_{2}]}m_{[i_{1},n]}+\cdots,
    \end{align*}
    \item[{\bf b.}] \[m_{[i_{2},j_{2}]}m_{[j_{2}+1,j_{2}+1]}=-m_{[i_{2},j_{2}]}m_{[i_{1},n]}+\cdots,\]
                  \[m_{[i_{2}-1,i_{2}-1]}m_{[i_{2},j_{2}]}=-m_{[i_{2},j_{2}]}m_{[i_{1},n]}+\cdots,\]
    \vspace{0.1cm}

\end{itemize}
 Thus, the coefficient of each {Type \RomanNumeralCaps 3} term is $2(j_{2}-i_{2}).$

 \item For {Type \RomanNumeralCaps 4} and {Type \RomanNumeralCaps 5} terms, we have \begin{itemize}
    \item[{\bf a.}] When $i_{1}>i+1,$ there is no term that involves  $k_{i}m_{[i_{1},n]}.$
    \item[{\bf b.}] When $i_{1}=i+1$, we have \[m_{[i,i]}m_{[i_{1},n]}+m_{[i,i]}m_{[i+1,i+1]}=0(k_{i}m_{[i_{1},n]})+\cdots.\]
    \item[{\bf c.}] When $i_{1}= i,$ the only term that involves  $k_{i}m_{[i_{1},n]}$ is \[m_{[i,i]}m_{[i+1,i+1]}=-k_{i}m_{[i_{1},n]}+\cdots.\]
    \item[{\bf d.}] When $n>i>i_{1}$, the following term involves  $k_{i}m_{[i_{1},n]}$:
     \[m_{[i-1,i-1]}m_{[i,i]}+m_{[i,i]}m_{[i+1,i+1]}=-2k_{i}m_{[i_{1},n]}+\cdots,\]
   \item[{\bf e.}] When $i=n$, the only term that involves with $k_{i}m_{[i_{1},n]}$ is \[m_{[n-1,n-1]}m_{[n,n]}=-k_{i}m_{[i_{1},n]}+\cdots.\]
\end{itemize}

 \item For {Type \RomanNumeralCaps 6} term, we have \begin{align*}\sum_{s=i}^{n-1}m_{[s,s]}m_{[s+1,s+1]}=(n-i)m^{2}_{[i,n]}+\cdots.\end{align*} So the coefficient of $m^{2}_{[i,n]}$ is $n-i.$

\end{itemize}

 By induction hypothesis and Proposition \ref{Rem Co}, we see that,
  $$B'({\bf m})-\frac{1}{2}{\mathbf{k}A\mathbf{k}^\top}={\rm codim}(\eta).$$ Thus, two $q$-series identities are equivalent.
\end{proof}

\begin{rem} It was briefly mentioned in  \cite{rimanyi2018partition} that Keller's quantum dilogarithm identity for type $A$ quivers \cite{keller2011cluster} is closely related to Theorem \ref{rr}. This section can be viewed as a precise clarification of that claim at least for one particular orientation of the quiver. 
\end{rem}

\section{Level one principal subspace of $B_2$}

In this and next section we consider principal subspaces of two level one representations of affine Lie algebras
for which we have a well-known spinor realization.
For more about level one spinor realization for $B_{\ell}^{(1)}$ and $D_{\ell}^{(1)}$ affine Kac-Moody Lie algebras
we refer the reader to \cite{feingold1985classical}.

\subsection{Spinor representation of $B_\ell^{(1)}$}

Level one affine vertex algebra of type $B_{\ell}^{(1)}$, denoted by $L_1(\so_{2\ell+1})$,  has a spinor realization via $2 \ell+1$ fermions.
We take
$$\mathcal{F}_{\ell}(Z+1/2):=\Lambda(a_i(-1/2),a_i(-3/2),\cdots ,a_i^*(-1/2),a_i^*(-3/2),...; 1 \leq i \leq \ell)$$
and $\mathcal{F}=\Lambda(e(-1/2),e(-3/2),...)$; here $\Lambda( \ \cdot \ )$ denotes the exterior algebra as in \cite{feingold1985classical}. Then the even part of
the vertex superalgebra
$$\mathcal{F}_{\ell}(Z+1/2) \otimes \mathcal{F}$$
is isomorphic to the affine vertex algebra $L_{1}(\so_{2 \ell+1})$. It is easy to see from the realization that the  principal subspace $W(\Lambda_0) \subset L_1(\so_{2\ell+1})$ is strongly generated by the following fields
$$ : a_i a_j :, \ \ \ 1 \leq i < j \leq \ell  \ \ \ \   : a_i a_j^* : \ \ \  1 \leq i < j , \ \ \  \  : a_i e:, 1 \leq i \leq \ell,$$
Note that we have $\ell(\ell-1)/2+\ell(\ell-1)/2+ \ell=\ell^2$ generators,  the dimension of $\frak{ n}_+.$

\subsection{Principal subspace of $L_1(\so_{5})$}
For $B_2^{(1)}$, the principal subspace $W(\Lambda_0)$ is generated by
$$ :a_1 a_2: ,  \ \: a_1 a_2^*:, \ \  :a_1 e:, \ \  :a_2 e: $$
corresponding to positive roots $\epsilon_1+\epsilon_2, \epsilon_1-\epsilon_2, \epsilon_1,\epsilon_2$, respectively. Its $C_{2}$-algebra, $R_{W(\Lambda_0)}$, is generated by $x_{\epsilon_1+\epsilon_2},x_{\epsilon_1-\epsilon_2}, x_{\epsilon_2},x_{\epsilon_1}$, and we have the following relations in $R_{W(\Lambda_0)}$:
\begin{align*}
& x_{\epsilon_1+\epsilon_2}^2=0,\quad x_{\epsilon_1-\epsilon_2}^2=0,\quad x_{\epsilon_1}^2-x_{\epsilon_1+\epsilon_2}x_{\epsilon_1-\epsilon_2}=0,\quad x_{\epsilon_2} x_{\epsilon_1+\epsilon_2}=0,\\ & \quad x_{\epsilon_1} x_{\epsilon_1-\epsilon_2}=0,\quad \quad x_{\epsilon_1} x_{\epsilon_1+\epsilon_2}=0,\quad x_{\epsilon_2}^3=0,\quad  x_{\epsilon_2}^2 x_{\epsilon_1}=0.
\end{align*}


We denote by $A$ the commutative algebra
\begin{align*} \frac{\mathbb{C}[x_{\epsilon_1+\epsilon_2},x_{\epsilon_1-\epsilon_2}, x_{\epsilon_2},x_{\epsilon_1}
]}{( x_{\epsilon_1+\epsilon_2}^2, x_{\epsilon_1-\epsilon_2}^2,x_{\epsilon_1}^2-x_{\epsilon_1+\epsilon_2}x_{\epsilon_1-\epsilon_2},x_{\epsilon_2} x_{\epsilon_1+\epsilon_2}, x_{\epsilon_1} x_{\epsilon_1-\epsilon_2}, x_{\epsilon_1} x_{\epsilon_1+\epsilon_2},x_{\epsilon_2}^3, x_{\epsilon_2}^2 x_{\epsilon_1})}. \end{align*}

\subsection{Character formula} It is known that the character of the principal subspace of $ L_{1}(\so_5)$ is
given by \cite{butorac2014combinatorial}:\begin{align}\label{char B}{\rm ch}[W(\Lambda_0)]=\sum_{r_1,r_2,r_3 \geq 0} \frac{q^{r_1^2+(r_2+r_3)^2+r_3^2-r_1(r_2+2 r_3)}}{(q)_{r_1}(q)_{r_2}(q)_{r_3}}.\end{align}
This identity is more complicated because $x_{\alpha_1}^3(z)=0$ but $x_{\alpha_1}^2(z) \neq 0$ for the short root $\alpha_1$. Observe that we do not need to impose
$x_{\epsilon_1}^3=0$ as it follows from other relations.

Let us introduce a filtration, $G$, on $J_{\infty}(A)$  by letting $G_{0}$  be  generated with $x_{\epsilon_1+\epsilon_2},x_{\epsilon_1-\epsilon_2},$ and (for $s \geq 1$)  \[G_{s}={\rm span}\left\{(x_{u})_{(i)}v| u\in\left\{\epsilon_{1},\epsilon_{2} \right\},v\in G_{s-1} \right\}+G_{s-1}.\]

Then we have the following lemma.
\begin{lem}\label{HSA}
The Hilbert series of the jet algebra of
\begin{align}\label{HS B} B:=\frac{\mathbb{C}[x_{\epsilon_1+\epsilon_2},x_{\epsilon_1-\epsilon_2}, x_{\epsilon_2},x_{\epsilon_1}
]}{( x_{\epsilon_1+\epsilon_2}^2, x_{\epsilon_1-\epsilon_2}^2,x_{\epsilon_1}^2,x_{\epsilon_2} x_{\epsilon_1+\epsilon_2}, x_{\epsilon_1} x_{\epsilon_1-\epsilon_2}, x_{\epsilon_1} x_{\epsilon_1+\epsilon_2},x_{\epsilon_2}^3, x_{\epsilon_2}^2 x_{\epsilon_1})}, \end{align}
is greater than or equal to   $HS_{q}(gr^{G}(J_{\infty}(A)))$.
\end{lem}

\begin{proof}
Since both  $J_{\infty}(B)$ and  $gr^{G}(J_{\infty}(A))$ are generated by $x_{\epsilon_1+\epsilon_2},x_{\epsilon_1-\epsilon_2}, x_{\epsilon_2},x_{\epsilon_1}$,  \color{black} it is enough to show that the quotient relations of $J_{\infty}(B)$ also hold in $gr^{G}(J_{\infty}(A))$. 
Note
$$( x_{\epsilon_1+\epsilon_2}^2, x_{\epsilon_1-\epsilon_2}^2,x_{\epsilon_2} x_{\epsilon_1+\epsilon_2}, x_{\epsilon_1} x_{\epsilon_1-\epsilon_2}, x_{\epsilon_1} x_{\epsilon_1+\epsilon_2},x_{\epsilon_2}^3, x_{\epsilon_2}^2 x_{\epsilon_1})_{\partial}$$ are quotient relations of $gr^{G}(J_{\infty}(A))$. For the quotient relations of $J_{\infty}(A)$, $(x_{\epsilon_1}^2-x_{\epsilon_1+\epsilon_2}x_{\epsilon_1-\epsilon_2})_{\partial}$, we have \[(x_{\epsilon_1}^2)_{\partial}\subset  G_{2}\setminus G_{0},\,\,\,\text{while}\,\,\,(x_{\epsilon_1+\epsilon_2}x_{\epsilon_1-\epsilon_2})_{\partial} \subset G_{0}.\]
Therefore, quotient relations, $(x_{\epsilon_1}^2-x_{\epsilon_1+\epsilon_2}x_{\epsilon_1-\epsilon_2})_{\partial}$, in $J_{\infty}(A)$ give us $(x_{\epsilon_1}^2)_{\partial}$ in $gr^{G}(J_{\infty}(A))$. The result follows.
\end{proof}



\begin{prop} \label{upper-B}
The Hilbert series of $J_\infty(B)$  is bounded from above by   \[\sum_{n_1,n_2,n_3,n_4,n_5 \geq 0} \frac{q^{n_1^2+n_2^2+(n_3+n_5)^2+n_4^2+n_3^2+(2 n_3+n_5)(n_1)+n_4(n_1+n_2)+n_3 n_4}}{(q)_{n_1}(q)_{n_2}(q)_{n_3}(q)_{n_4}(q)_{n_5}},\]
in the sense that all $q$-coefficients of $HS_q(J_\infty(B))$ are less than or equal to the corresponding $q$-coefficients of this $q$-series.
\end{prop}
\begin{proof} We will prove that $J_\infty(B)$ is spanned by a set of monomials whose character is given by the
$q$-series in the statement.
We first observe that we can filter $J_\infty(B)$ with the number of "particles" of type $x_{\epsilon_2}$.
In the zero component of the filtration we have the jet algebra with relations
\[ ( x_{\epsilon_1+\epsilon_2}^2, x_{\epsilon_1-\epsilon_2}^2,x_{\epsilon_1}^2, x_{\epsilon_1} x_{\epsilon_1-\epsilon_2}, x_{\epsilon_1} x_{\epsilon_1+\epsilon_2})_{\partial} \]
whose Hilbert series is given by \cite{li2020some}
\[\sum_{n_1,n_2,n_4 \geq 0} \frac{q^{n_1^2+n_2^2+n_4^2+n_4(n_1+n_2)}}{(q)_{n_1}(q)_{n_2}(q)_{n_4}}.\]
Now we include $x_{\epsilon_2}$ generator and additional relations. It is know that the jet algebra of  $\mathbb{C}[x_{\epsilon_2}]/(x_{\epsilon_2}^3)$ admits a combinatorial basis satisfying difference two at the distance two condition as in Rogers-Selberg identities. Therefore, its Hilbert series is given by
\[\sum_{n_3,n_5 \geq 0} \frac{q^{(n_3+n_5)^2+n_3^3}}{(q)_{n_3}(q)_{n_5}}.\]
This formula can be also explained using jet schemes by letting $a=x_{\epsilon_2}^2$ and $b=x_{\epsilon_2}$
and considering $(ab,a-b^2,a^2)_{\partial}$. Since $a$ is of degree two, we get contribution $2n_3^2+2n_3 n_5+n_5^2$ in the exponent.

To finish the proof we have to analyze two additional relations and their contribution:
$$x_{\epsilon_2} x_{\epsilon_1+\epsilon_2}=0 \ \ {\rm and} \ \ x_{\epsilon_2}^2 x_{\epsilon_1}=0.$$
The first relation contributes with the boundary condition $(2n_3+n_5)n_1$ in the $q$-exponent as $(2n_3+n_5)$ counts the power
of $x_{\epsilon_2}$ in $(x_{\epsilon_2}^2)^{n_3} x_{\epsilon_2}^{n_5}$ and we have relation $x_{\epsilon_2} x_{\epsilon_1+\epsilon_2}=0$ ($n_1$ summation variable corresponds to $x_{\epsilon_1+\epsilon_2}$).
For the second relation, using $x_{\epsilon_2} x_{\epsilon_1} \neq 0$ and $x_{\epsilon_2}^2  x_{\epsilon_1}=0$,
we get a spanning set of monomials involving $a=x_{\epsilon_2}^2$ and $x_{\epsilon_{1}}$  to be
\begin{align*}
   \left\{a_{(-n_{i})}\cdots a_{(-n_{1})} (x_{\epsilon_{1}})_{(m_{j})}\cdots  (x_{\epsilon_{1}})_{(m_{1})}|j+1\leq
   n_{1},n_{s}-n_{s-1}\geq 4, m_{s}-m_{s-1}\geq 2\right\},
\end{align*}
where $j+1\leq n_{1}$ (boundary condition), $n_{s}-n_{s-1}\geq 4, m_{s}-m_{s-1}\geq 2$ (difference condition) are coming from $x_{\epsilon_2}^{2} x_{\epsilon_1}=0$, $x_{\epsilon_{2}}^{3}=0$ and $x_{\epsilon_{1}}^{2}=0$, respectively. Then the character of the space spanned by these monomials is at most  \[\sum_{n_3,n_4 \geq 0} \frac{q^{n_4^2+2n_3^2+n_3 n_4}}{(q)_{n_3}(q)_{n_4}}.\]
By combining these arguments together, we obtain the result.
\end{proof}

Next formula was first discussed in \cite[formula (4.5)]{Andrews}:
\begin{align} \label{an-ide}
&  \sum_{\substack{r_1,r_2,r_3 \geq 0}} \frac{q^{r_1^2+(r_2+r_3)^2+r_3^2-r_1(r_2+2 r_3)}}{(q)_{r_1}(q)_{r_2}(q)_{r_3}} \\
& = \left( \sum_{\substack{r_1,r_2,r_3 \geq 0}} \frac{q^{r_1^2+r_2^2+r_3^2+r_1 r_2+r_2 r_3}}{(q)_{r_1}(q)_{r_2}(q)_{r_3}}\right) \left( \sum_{n \geq 0} \frac{q^{n^2}}{(q)_n} \right)=
\frac{(-q;q)_\infty (-q;q^2)^2_\infty}{(q,q^4;q^5)_\infty}. \nonumber
\end{align}
This formula, in particular implies an interesting
relationship between character formulas of three different types of level one principal subspaces (to avoid confusion we added subscripts to indicate the underlying Lie algebra):  \begin{equation}
{\rm ch}[W_{B_2}(\Lambda_0)](\tau)={\rm ch}[W_{A_2}(\Lambda_0)](\tau) \cdot {\rm ch}[W_{A_1}(\Lambda_0)](\tau).
\end{equation}
Our next result shows that the upper bound obtained in Proposition \ref{upper-B} is in fact the character.
\begin{thm}\label{Btype}
We have the identity
\begin{align} \label{charge-id} &\sum_{\substack{r_1,r_2,r_3 \geq 0}} y_1^{r_1} y_2^{2r_3+r_2} \frac{q^{r_1^2+(r_2+r_3)^2+r_3^2-r_1(r_2+2 r_3)}}{(q)_{r_1}(q)_{r_2}(q)_{r_3}}= \sum_{\substack{n_1,n_2,n_3,n_4,n_{5} \geq 0}}\\  &
 y_1^{n_1+n_2+n_4}  y_2^{2n_1+2n_3+n_4+n_5}  \frac{q^{n_1^2+n_2^2+(n_3+n_5)^2+n_4^2+n_3^2+(2 n_3+n_5)(n_1)+n_4(n_1+n_2)+n_3 n_4}}{(q)_{n_1}(q)_{n_2}(q)_{n_3}(q)_{n_4}(q)_{n_5}}
 \nonumber  \end{align}
\end{thm}
\begin{proof} We first prove the identity specialized at $y_1=y_2=1$. We recall a formula from \cite[Lemma 1]{Andrews} (attributed to Askey):
\begin{align*}
& \sum_{n_1,...,n_r \geq 0 \atop m_1,...m_s,v \geq 0} \frac{q^{Q_1(n_1+v,n_2,..,n_r)+Q_2(m_1+v,m_2,...,m_s)+m_1 n_1}}{(q)_{n_1} \cdots (q)_{n_r}(q)_{m_1} \cdots (q)_{m_s}(q)_v} \\
& =\left(\sum_{n_1,,,,,n_r \geq 0} \frac{q^{Q_1(n_1,...,n_r)}}{(q)_{n_1} \cdots (q)_{n_r}} \right) \left(\sum_{m_1,...,m_s \geq 0} \frac{q^{Q_1(m_1,...,m_s)}}{(q)_{m_1} \cdots (q)_{m_s}} \right),
\end{align*}
where $Q_1$ and $Q_2$ are quadratic forms. Specialize now, with a more convenient choice of indices, $Q_1(n_1,n_4,n_2)=n_1^2+n_4^2+n_2^2+n_1 n_4 +n_4 n_2, Q_2(n_5)=n_5^2$.
Then the relation
$$Q_1(n_1+n_3,n_4,n_2)+Q_2(n_5+n_3)+n_5 n_1$$
$$=n_1^2 + n_2^2 + (n_3 + n_5)^2 + n_4^2 + n_3^2 + (2 n_3 + n_5) n_1 +
 n_4(n_1 + n_2) + n_3 n_4$$
together with (\ref{an-ide}) now gives the specialized formula.
To prove the formula with charge variables $y_1$ and $y_2$ it is sufficient to observe that our argument in Proposition \ref{upper-B} is also valid with the charge variables.
This means that the left-hand side in (\ref{charge-id}) is less than or equal the right-hand side for any coefficient $y_1^{m} y_2^{n}$. Since these expressions are
equal when $y_1=y_2=1$, all coefficients $y_1^{m} y_2^{n}$ must be equal on both sides so we have the claimed relation.
\end{proof}
\begin{cor}
The $C_{2}$-algebra of the principal subspace of $L_{1}({\so_5})$, $R_{W(\Lambda_0)}$, is isomorphic to $A$, and the principal subspace $W(\Lambda_0)$ is classically free.
\end{cor}
\begin{proof}
As in Section 6.3 we can use the same filtration $G$ on $J_{\infty}(R_{W(\Lambda_0)})$. Then using the surjectivity of $\psi$ and previous discussion we have \begin{align*} &{\rm ch}[W(\Lambda_0)](q) \leq HS_{q}(J_{\infty}(R_{W(\Lambda_0)}))= HS_{q}(gr^{G}(J_{\infty}(R_{W(\Lambda_0)})))\\&\leq HS_{q}(gr^{G}(J_{\infty}(A)). \end{align*}
Also from Lemma \ref{HSA}, we get
$$HS_{q}(gr^{G}(J_{\infty}(A)) \leq HS_q(J_{\infty}(B)) = {\rm ch}[W(\Lambda_0)],$$
where in the last equation we use Theorem \ref{Btype} (specialized at $y_1=y_2=1$).
By combining these inequalities we get that the character and all intermediate Hilbert series of jet algebras are equal, therefore,
${\rm ch}[W(\Lambda_0)](q)=HS_{q}(J_{\infty}(R_{W(\Lambda_0)}))$
and $\psi$ is injective.
\end{proof}

\section{Level one principal subspace of $D_4$}

\subsection{Spinor representation of $D_\ell^{(1)}$}

Level one affine vertex algebra $L_1(\so_{2\ell})$ has a spinor realization via $2 \ell$ fermions \cite{feingold1985classical}.
We take
$$\mathcal{F}_{\ell}(Z+1/2):=\Lambda(a_i(-1/2),a_i(-3/2),\cdots ,a_i^*(-1/2),a_i^*(-3/2),...; 1 \leq i \leq \ell).$$ Then the even part of
$$\mathcal{F}_{\ell}(Z+1/2)$$
is isomorphic to $D_\ell^{(1)}$-module $L_{\so_{2\ell}}(\Lambda_0)$, denoted by $L_1(\so_{2 \ell})$ when viewed as a vertex algebra. The  principal subspace $W(\Lambda_0) \subset L_1(\so_{2 \ell})$ is strongly generated by
$$ : a_i a_j :, \ \ \ 1 \leq i < j \leq \ell  \ \ \ \   : a_i a_j^* : \ \ \  1 \leq i < j\leq \ell , $$
all together $\ell(\ell-1)=\ell^2-\ell$ root vectors.

For $D_{4}^{(1)}$, the principal subspace is generated by

$$ :a_i a_j  \\: a_i a_j^*: \\ (1\leq i<j\leq 4), $$
corresponding to $\epsilon_i+\epsilon_j, \epsilon_i-\epsilon_j$, respectively. And its $C_{2}$-algebra, $R_{W(\Lambda_0)}$, is generated by $x_{\epsilon_i+\epsilon_j},x_{\epsilon_i-\epsilon_j}$ $(1\leq i<j\leq 4)$, which we denote by $W_{ij}$ and $V_{ij}$, respectively. And we have the following relations in $R_{W(\Lambda_0)}$:


\begin{align*}
& W_{12}W_{23},\quad W_{12}W_{24},\quad W_{13}W_{34},\quad W_{23}W_{34},\quad W_{12}W_{12},\quad W_{12}W_{13},\quad W_{12}W_{14}, \quad \\ &  W_{23}W_{23}, \quad W_{23}W_{24},\quad W_{24}W_{24},\quad W_{34}W_{34},\quad W_{13}W_{23},\quad W_{14}W_{24},\quad W_{14}W_{34},\quad \\ & W_{24}W_{34},\quad -W_{13}W_{24}=W_{23}W_{14}=W_{12}W_{34}, \quad W_{12}V_{23},\quad W_{12}V_{24},\quad W_{13}V_{34},\quad \\ & W_{23}V_{34},\quad W_{12}V_{13},\quad W_{13}V_{12},\quad W_{12}V_{14},\quad W_{14}V_{12},\quad W_{23}V_{24},W_{24}V_{23},\quad\\&  W_{13}V_{14},\quad W_{14}V_{13},\quad W_{12}V_{12}=W_{13}V_{13}=W_{14}V_{14},\quad W_{23}W_{23}=W_{24}W_{24},\quad \\& -W_{13}V_{24}=W_{23}V_{14}=W_{12}V_{34},\quad -W_{14}V_{23}=W_{24}V_{13},\quad V_{12}V_{12},\quad V_{12}V_{13},\quad V_{12}V_{14},\quad \\ & V_{23}V_{23},\quad V_{23}V_{24},\quad V_{24}V_{24},\quad V_{34}V_{34},\quad V_{13}V_{23},\quad V_{14}V_{24},\quad V_{14}V_{34},\quad V_{24}V_{34}, \quad\\&  -V_{13}V_{24}=V_{23}V_{14},\quad W_{13}V_{23}+W_{23}V_{13}=W_{24}V_{14}+W_{14}V_{24}.\end{align*} We denote by $D$ the algebra $\mathbb{C}[W_{ij},V_{ij}|1\leq i<j\leq 4]/I$, where $I$ is generated by above quadratic relations. We expect that  the algebra $D$ is the $C_{2}$-algebra of $R_{W(\Lambda_0)}.$

\subsection{Jet algebra and quantum dilogarithm}
We first assume \[x_{1}x_{2}=qx_{2}x_{1},\; x_{2}x_{3}=qx_{3}x_{2},\; x_{2}x_{4}=qx_{4}x_{2}.\] Then by using properties of quantum dilogarithm we have  \begin{align*}& \phi(-q^{\frac{1}{2}}x_{4})\phi(-q^{\frac{1}{2}}x_{3})\phi(-q^{\frac{1}{2}}x_{2})\phi(-q^{\frac{1}{2}}x_{1})\\ =& \phi(-q^{\frac{1}{2}}x_{1})\phi(-q x_{2}x_{1})\phi(-q^{\frac{3}{2}}x_{4}x_{2}x_{1})\phi(-q^{\frac{3}{2}}x_{3}x_{2}x_{1})\phi(-q^{\frac{1}{2}}x_{2})\phi(-q^{\frac{5}{2}}x_{4}x_{3}x_{2}x_{1}x_{2})\\ &\phi(-q^{2}x_{4}x_{3}x_{2}x_{1}) \phi(-q x_{4}x_{2})\phi(-q x_{3}x_{2})\phi(-q^{\frac{3}{2}}x_{4}x_{3}x_{2})\phi(-q^{\frac{1}{2}}x_{3})\phi(-q^{\frac{1}{2}}x_{4}).\end{align*} It is equivalent with \begin{align*}\sum_{{\bf (m,n)}\in\mathbb{N}^{12}}\frac{q^{B({\bf m,n})}x_{1}^{\lambda_{1}}x_{2}^{\lambda_{2}}x_{3}^{\lambda_{3}}x_{4}^{\lambda_{4}}}{\prod_{1\leq i<j\leq n}(q)_{m_{ij}}(q)_{n_{ij}}}=\sum_{{\bf k}\in (k_{1},k_{2}k_{3},k_{4})\in \mathbb{N}^{4}}\frac{q^{\frac12 {\bf k}\mathcal{D}{\bf k}^{\top}}x_{1}^{k_{1}}x_{2}^{k_{2}}x_{3}^{k_{3}}x_{4}^{k_{4}}}{(q)_{k_{1}}(q)_{k_{2}}(q)_{k_{3}}(q)_{k_{4}}}, \end{align*} where \begin{itemize}
    \item $\mathcal{D}$ is the Cartan matrix of type $D_{4},$
    \item ${\bf (m,n)}=(m_{ij},n_{ij}|1\leq i<j\leq n)$, \item \begin{align*} &\lambda_{1}=m_{12}+m_{13}+m_{14}+n_{14}+n_{13}+n_{12},\\
    &\lambda_{2}=m_{23}+m_{14}+m_{13}+m_{24}+n_{24}+n_{23}+n_{13}+2n_{12}+n_{14},\\
    &\lambda_{3}=m_{34}+m_{14}+m_{24}+n_{23}+n_{13}+n_{12},\\
    &\lambda_{4}=n_{34}+n_{14}+n_{24}+n_{23}+n_{13}+n_{12},
    \end{align*}
    \item \begin{align*}B({\bf m,n})=&{m_{12}}^2+{m_{12}} {m_{13}}+{m_{12}} {m_{14}}+{m_{12}}
   {n_{12}}+{m_{12}} {n_{13}}+{m_{12}} {n_{14}}+{m_{13}}^2\\ &+{m_{13}}
   {m_{14}}+{m_{13}} {m_{23}}+{m_{13}} {m_{24}}+2 {m_{13}}
  {n_{12}}+{m_{13}} {n_{13}}+{m_{13}} {n_{14}}\\ &+{m_{13}}
   {n_{23}}+{m_{13}} {n_{24}}+{m_{14}}^2+{m_{14}} {m_{24}}+{m_{14}}
   {m_{34}}+{m_{14}} {n_{12}}+{m_{14}} {n_{13}}\\ &+{m_{14}}
   {n_{23}} +{m_{23}}^2+{m_{23}} {m_{24}}+{m_{23}} {n_{12}}+{m_{23}}
   {n_{23}}+{m_{23}} {n_{24}}+{m_{24}}^2\\ & +{m_{24}} {m_{34}}+{m_{24}}
  {n_{12}}+{m_{24}} {n_{23}}+{m_{34}}^2+{m_{34}} {n_{12}}+{m_{34}}
  {n_{13}}+{m_{34}} {n_{23}}\\ & +{n_{12}}^2+{n_{12}} {n_{13}}+{n_{12}}
   {n_{14}}+2 {n_{12}} {n_{23}}+{n_{12}}{n_{24}}+{n_{12}}
  {n_{34}}+{n_{13}}^2\\ &+{n_{13}} {n_{14}}+{n_{13}} {n_{23}}+{n_{13}}
   {n_{34}}+{n_{14}}^2+{n_{14}} {n_{23}}+{n_{14}} {n_{24}}+{n_{14}}
   {n_{34}}+{n_{23}}^2\\ &+{n_{23}} {n_{24}}+{n_{23}}
  {n_{34}}+{n_{24}}^2+{n_{24}} {n_{34}}+{n_{34}}^2.\end{align*}
\end{itemize}

Note the right-hand side of above identity is the character of the principal subspace of $L_1(\so_{8})$. Then we have:
\begin{prop}
The following two statements are equivalent:
\begin{itemize}
    \item[{\bf \romannumeral 1}] The Hilbert series of $J_{\infty}(R_{W(\Lambda_{0}}))$ equals \[\sum_{{\bf (m,n)}\in\mathbb{N}^{12}}\frac{q^{B({\bf m,n})}}{\prod_{1\leq i<j\leq n}(q)_{m_{ij}}(q)_{n_{ij}}}.\]
    \item[{\bf \romannumeral 2}] The principal subspace $W(\Lambda_0)$ of $L_{1}(\so_{8})$ is classically free.
\end{itemize}
\end{prop}

\begin{rem}
 For algebra $D$, we can choose filtration $G^{1}$ and $G^{2}$ by letting $G^{1}_{0}$ and $G^{2}_{0}$ be generated with $\left\{ V_{23},V_{24}\right\}$ and $\left\{V_{14},W_{24}\right\}$, respectively. And let \[G^{1}_{s}={\rm span}\left\{(W_{ij})_{(n)}v,(V_{ij})_{(n)}v|n\leq 1, V_{ij}\neq V_{23},\, \text{or}\, V_{24} , v\in G^{1}_{s-1}\right\}+G^{1}_{s-1},\] and  \[G^{2}_{s}={\rm span}\left\{(W_{ij})_{(n)}v,(V_{ij})_{(n)}v|n\leq 1, V_{ij}\neq V_{14}, W_{ij}\neq W_{24} , v\in G^{2}_{s-1}\right\}+G^{2}_{s-1}.\] Then following the similar argument in Proposition \ref{filtra}, it is not hard to see that $HS_{q}(J_{\infty}(D))$ is less than or equal to   \[\sum_{{\bf (m,n)}\in\mathbb{N}^{12}}\frac{q^{B'({\bf m,n})}}{\prod_{1\leq i<j\leq n}(q)_{m_{ij}}(q)_{n_{ij}}},\] where $B'({\bf m,n})$ differs from $B({\bf m,n})$ by $n_{12}n_{23}$ and  $n_{12}m_{13},$ i.e., $$B({\bf m,n})-B'({\bf m,n})=n_{12}n_{23}+n_{12}m_{13}.$$   The method we have used to prove the classically freenes in the case of $\fsl_{n}$  fails here, since the "filtration procedure" does not give us a satisfactory upper bound of $HS_{q}(J_{\infty}(D))$.
\end{rem}

In \cite{feigin2010zhu}, a characterization of the $C_{2}$-algebra of the affine vertex algebra of type $C$ at non-negative integer level  $k$ was given. For orthogonal series, a certain quotient of the $C_{2}$-algebra of the affine vertex algebra of type $D$ at non-negative level was also determined. We hope to use their results to further investigate the freeness of the affine vertex algebra of type $C$ and $D$ in the future.

\section{Level one principal subspace of $G_2$}

In this part we consider the principal subspace $W(\Lambda_0)$ associated to the exceptional Lie algebra of type $G_2$.  
We denote the root system of type $G_2$ by $\Delta$ and let $\Delta_+=\{ \alpha_1,\alpha_2, \alpha_1+\alpha_2, \alpha_1+2\alpha, \alpha_1+3 \alpha_2,2 \alpha_1+3 \alpha_2 \}$ be the 
set of positive roots. In our computations, we choose a specific realization of the exceptional Lie algebra $\frak{g}_2$ where the root vectors $x_\alpha$, $\alpha \in \Delta$ are chosen as in \cite{Hesselink}. 
Then as before we consider $\frak{n}={\rm Span} \{ x_{\alpha}: \alpha \in \Delta_+ \}$ and construct the principal subspace vertex algebra $W(\Lambda_0) \subset L_1(\frak{g}_2)$.
Let
$$x:=x_{\alpha_1},y:=x_{\alpha_1+3 \alpha_2},z:=x_{2\alpha_1+3 \alpha_2},r := x_{\alpha_2},s:=x_{\alpha_1+ \alpha_2},v:=x_{\alpha_1+2 \alpha_2},$$
viewed now as elements $R_{W(\Lambda_0)}$.
Then we have the following relations in $R_{W(\Lambda_0)}$ (and also in $R_{L(\Lambda_0)})$, which follow by straightforward computation:
$$x^{2}=0,\,y^{2}=0,\,z^{2}=0,\,y\,z=0,\,z\,v=0,\,y\,v=0,\,2\,y\,s+v^{2}=0,\,y\,r=0,\,s^{2}-2\,x\,v=0,$$
$$\,x\,s=0,\,4\,z\,r-2\,v^{2}=0,\,z\,s=0,\,x\,z=0,\,-x\,y+s\,v=0,\,r^{4}=0,$$
$$\,s\,v^{2}=0,\,r^{2}v=0,\,r^{3}s=0 $$
Denote by $I \subset \mathbb{C}[x,y,z,r,s,v]$ the ideal generated by the left-hand sides of the relations above.
Then we expect that 
$$J_\infty(R_{W(\Lambda_0)}) \cong J_{\infty}(\mathbb{C}[x,y,z,r,s,v]/I) \cong {\rm gr}(W(\Lambda_0)).$$
These isomorphisms imply:
\begin{conj} We have an equality of $q$-series
$$HS_q(J_{R_{W(\Lambda_0)}}) ={\rm ch}[W(\Lambda_0)].$$
In particular, $W(\Lambda_0)$ is classically free.
\end{conj}
We verified the conjecture modulo $O(q^5)$ using Macaulay2.








\vskip 5mm
\bibliography{identity}
\bibliographystyle{alpha}

\end{document}